\documentclass{siamltex}
\usepackage{amsfonts,amsmath,amssymb}
\usepackage{mathrsfs,mathtools}
\usepackage{enumerate}
\usepackage{xspace,MACROS/mydef2}
\usepackage{hyperref}
\usepackage{esint}
\usepackage{graphicx,epstopdf}
\usepackage{textcomp}
\usepackage{calc}
\usepackage{nicefrac}
\usepackage{stmaryrd}
\usepackage{bbm}
\DeclareMathAlphabet{\mathpzc}{OT1}{pzc}{m}{it}

\newtheorem{remark}[theorem]{Remark}
\numberwithin{equation}{section}

\newcommand{\Nin}{\,{\mbox{\,\raisebox{6.0pt} {\tiny$\circ$} \kern-11.1pt}\N }}
\newcommand{\Ninn}{{\mbox{\,\raisebox{4.5pt} {\tiny$\circ$} \kern-8.8pt}\N }}
\newcommand{\osc}{{\textup{\textsf{osc}}}}

\title{An a posteriori error analysis for an optimal control problem involving the fractional Laplacian\thanks{The first author's research has been partially supported by the NSF grant DMS-1521590. The second author's research has been partially supported by CONICYT through FONDECYT project 3160201.}}

\author{Harbir Antil\thanks{Department of Mathematical Sciences,
George Mason University, Fairfax, VA 22030, USA. 
\texttt{hantil@gmu.edu}}.
\and
Enrique Ot\'arola\thanks{Departamento de Matem\'atica, 
Universidad T\'ecnica Federico Santa Mar\'ia, Valpara\'iso, Chile. 
\texttt{enrique.otarola@usm.cl}.}}

\date{Draft version of \today.}

\begin{document}

\maketitle
\begin{abstract}
In a previous work, we introduced a discretization scheme for a constrained optimal control problem involving the fractional Laplacian. For such a control problem, we derived optimal a priori error estimates that demand the convexity of the domain and some compatibility conditions on the data. 
To relax such restrictions, in this paper, we introduce and analyze an efficient and, under certain assumptions, reliable a posteriori error estimator.
We realize the fractional Laplacian as the Dirichlet-to-Neumann map for a nonuniformly elliptic problem posed on a semi--infinite cylinder in one more spatial dimension. This extra dimension further motivates the design of an posteriori error indicator. The latter is defined as the sum of three contributions, which come from the discretization of the state and adjoint equations and the control variable. The indicator for the state and adjoint equations relies on an anisotropic error estimator in Muckenhoupt weighted Sobolev spaces. The analysis is valid in any dimension. On the basis of the devised a posteriori error estimator, we design a simple adaptive strategy that exhibits optimal experimental rates of convergence.
\end{abstract}



\section{Introduction}
\label{sec:introduccion}
In this work we shall be interested in the derivation and analysis of a computable, efficient and, under certain assumptions, reliable a posteriori error estimator for a constrained linear-quadratic optimal control problem involving fractional powers of the Dirichlet Laplace operator. To the best of our knowledge, this is the first work that addresses this problem. To make matters precise, for $n\ge1$, we let $\Omega$ be an open and bounded polytopal domain of $\R^n$ with Lipschitz boundary $\partial \Omega$. Given $s \in (0,1)$, and a desired state $\usf_d: \Omega \rightarrow \mathbb{R}$, we define the cost functional
\begin{equation}
\label{def:J}
J(\usf,\zsf)= \frac{1}{2}\| \usf - \usf_{d} \|^2_{L^2(\Omega)} + 
\frac{\mu}{2} \| \zsf\|^2_{L^2(\Omega)},
\end{equation}
where $\mu > 0$ is the so-called regularization parameter. With these ingredients at hand, we define the \emph{fractional optimal control problem} as follows: Find
\begin{equation}
\label{def:minJ}
 \text{min }J(\usf,\zsf),
\end{equation}
subject to the \emph{fractional state equation}
\begin{equation}
\label{eq:fractional}
\Laps \usf = \zsf  \text{ in } \Omega, \qquad \usf = 0   \text{ on } \partial \Omega, \\
\end{equation}
and the \emph{control constraints}
\begin{equation}
 \label{eq:control_constraints}
\asf(x') \leq \zsf(x') \leq \bsf(x') \quad\textrm{a.e.~~} x' \in \Omega . 
\end{equation}
The operator $\Laps$, with $s \in (0,1)$, denotes the fractional powers of the Dirichlet Laplace operator, which for convenience we will simply call the \emph{fractional Laplacian}. The functions $\asf$ and $\bsf$ both belong to $L^2(\Omega)$ and satisfy the property $\asf(x') \leq \bsf(x')$ for almost every $x' \in \Omega$. 
 
A rather incomplete list of problems where fractional derivatives and fractional diffusion appears includes: mechanics \cite{atanackovic2014fractional}, where they are used to model viscoelastic behavior \cite{MR2035411}, turbulence \cite{wow,NEGRETE} and the hereditary properties of materials \cite{MR1926470}; diffusion processes \cite{Abe2005403,PSSB:PSSB2221330150}, in particular processes in disordered media, where the disorder may change the laws of Brownian motion and thus lead to anomalous diffusion \cite{MR1736459,MR1081295}; nonlocal electrostatics \cite{ICH}; finance \cite{MR2064019}; image processing \cite{GH:14}; biophysics \cite{bio}; chaotic dynamical systems \cite{MR1604710} and many others \cite{BV,MR2025566}. Optimal control problems arise naturally in these applications and then it is essential to design numerical schemes to efficiently approximate them.

The analysis of problems involving the fractional Laplacian is delicate and involves fine results in harmonic analysis \cite{Landkof,Silvestre:2007,Stein}; one of the main difficulties being the nonlocality of the operator. This difficulty has been resolved to some extent by L. Caffarelli and L. Silvestre \cite{CS:07}, who have proposed a technique that turned out to be a breakthrough and has paved the way to study fractional laplacians using \emph{local} techniques. Namely, any power $s \in (0,1)$ of the fractional Laplacian in $\R^n$ can be realized as an operator that maps a Dirichlet boundary condition to a Neumann-type condition via an extension problem on the upper half-space $\R_{+}^{n+1}$. This result was later adapted in \cite{CDDS:11,ST:10} to bounded domains $\Omega$, thus obtaining an extension problem posed on the semi-infinite cylinder $\C = \Omega \times (0,\infty)$. This extension corresponds to the following mixed boundary value problem:
\begin{equation}
\label{eq:alpha_harm}
  \DIV\left( y^\alpha \nabla \ue \right) = 0 \text{ in } \C, \quad \ue = 0 \text{ on } \partial_L \C, \quad
  \partial_{\nu^\alpha} \ue = d_s \zsf \quad \text{ on } \Omega \times \{0\},
\end{equation}
where $\partial_L \C= \partial \Omega \times [0,\infty)$ is the lateral boundary of $\C$ and $d_s = 2^{\alpha}\Gamma(1-s)/\Gamma(s)$ is a positive normalization constant. The parameter $\alpha$ is defined as $\alpha = 1-2s \in (-1,1)$ and the conormal exterior derivative of $\ue$ at $\Omega \times \{ 0 \}$ is
\begin{equation}
\label{def:lf}
\partial_{\nu^\alpha} \ue = -\lim_{y \rightarrow 0^+} y^\alpha \ue_y.
\end{equation}
We call $y$ the \emph{extended variable} and call the dimension $n+1$ in $\R_+^{n+1}$ the \emph{extended dimension} of problem \eqref{eq:alpha_harm}. The limit in \eqref{def:lf} must be understood in the distributional sense; see \cite{CS:07,CDDS:11,ST:10}. With these elements at hand, we then write the fundamental result by  L. Caffarelli and L. Silvestre \cite{CS:07,CDDS:11,ST:10}: the fractional Laplacian and the Dirichlet-to-Neumann map of problem \eqref{eq:alpha_harm} are related by 
$
  d_s \Laps \usf = \partial_{\nu^\alpha} \ue
$
in $\Omega$. 

The use of the aforementioned localization techniques for the numerical treatment of problem \eqref{eq:fractional} followed not so long after \cite{NOS}. In this reference, the authors propose the following technique to solve problem \eqref{eq:fractional}: given $\zsf$, solve \eqref{eq:alpha_harm}, thus obtaining a function $\ue$; setting $\usf(x') = \ue (x',0)$, the solution to \eqref{eq:fractional} is obtained. The implementation of this scheme uses standard components of finite element analysis, while its analysis combines asymptotic properties of Bessel functions \cite{Abra}, elements of harmonic analysis \cite{Javier, Muckenhoupt} and a polynomial interpolation theory on weighted spaces \cite{DL:05,NOS2}. The latter is valid for tensor product elements that exhibit a large aspect ratio in $y$ (anisotropy), which is necessary to fit the behavior of $\ue(x',y)$ with $x' \in \Omega$ and $y>0$. The main advantage of this scheme is that it solves the local problem \eqref{eq:alpha_harm} instead of dealing with $\Laps$ in \eqref{eq:fractional}. However, this comes at the expense of incorporating one more dimension to the problem; issue that has been resolved to some extent with the design of fast solvers \cite{CNOS} and adaptive finite element methods (AFEMs) \cite{CNOS2}.

Exploiting the ideas developed in \cite{NOS}, in the previous work \cite{AO}, we have proposed two numerical strategies to approximate the solution to \eqref{def:minJ}--\eqref{eq:control_constraints}. Invoking the localization results of \cite{CS:07,CDDS:11,ST:10}, we have considered an equivalent optimal control problem: $\min J(\ue(\cdot,0),\zsf)$ subject to the \emph{linear state equation} \eqref{eq:alpha_harm} and the \emph{control constraints} \eqref{eq:control_constraints}. Since \eqref{eq:alpha_harm} is posed on the semi-infinite cylinder $\C$, we have then introduced a truncated optimal control problem and analyzed its approximation properties. On the basis of this, we have proposed two schemes based on the discretization of the state and adjoint equations with first-degree tensor product finite elements on anisotropic meshes: the variational approach \cite{Hinze:05} and a fully discrete scheme that discretizes the set of controls by piecewise constant functions \cite{ACT:02,CT:05}. The latter yields an optimal error estimate for the control approximation: If $\Omega$ is convex, $\usf_d \in \mathbb{H}^{1-s}(\Omega)$, and $\asf, \bsf \in \mathbb{R}$ are such that $\asf \leq 0 \leq \bsf$ for $s \in (0,\tfrac{1}{2}]$, then
\begin{equation}
\label{eq:a_priori}
 \| \ozsf - \bar{Z} \|_{L^2(\Omega)} \lesssim | \log N |^{2s} N^{-\frac{1}{n+1}},
\end{equation}
where $\ozsf$ denotes the optimal solution to fractional optimal control problem, $\bar{Z}$ corresponds to the optimal solution of the discrete counterpart of \eqref{def:minJ}--\eqref{eq:control_constraints} and $N$ denotes the number of the degrees of freedom of the underlying mesh.

Since the aforementioned scheme incorporates one extra dimension, it raises the following question: How efficient is this method? A quest for an answer to this question motivates the study of AFEMs, since it is known that they constitute an efficient class of numerical methods for approximating the solution to optimal control problems \cite{MR1780911,HHIK,KRS}: they allow for their resolution with relatively modest computational resources. In addition, they can achieve optimal performance, measured as error versus degrees of freedom, in situations when classical FEM cannot \cite{KRS,NV,NSV:09}. An essential ingredient of AFEMs is an posteriori error estimator, which is a computable quantity that depends on the discrete solution and data, and provides information about the local quality of the approximate solution. For linear second-order elliptic boundary value problems, the theory has attained a mature understanding; see \cite{AObook,MR1770058,NSV:09,NV,Verfurth} for an up-to-date discussion including also the design of AFEMs, their convergence and optimal complexity. In contrast to this well-established theory, the a posteriori error analysis for constrained optimal control problem has not been fully understood yet; the main source of difficulty is its inherent nonlinear feature. We refer the reader to \cite{KRS} for an for an up-to-date discussion.

AFEMs for the fractional optimal control problem are also motivated by the fact that the a priori error estimate \eqref{eq:a_priori} requires $\ozsf \in \mathbb{H}^{1-s}(\Omega)$, which in turn demands $\Omega$ convex, $\usf_d \in \mathbb{H}^{1-s}(\Omega)$ and $\asf \leq 0 \leq \bsf$ for $s \in (0,\tfrac{1}{2}]$. If one of these conditions does not hold, the optimal control $\ozsf$ may have singularities in the $x'$-variables and thus exhibits fractional regularity. Consequently, quasi-uniform refinement of $\Omega$ would not result in an efficient solution technique; see \cite[section 6.3]{NOS} for an illustration of this situation at the level of solving the state equation \eqref{eq:alpha_harm}.

The main contribution of this work is the design and analysis of a computable, efficient and, under certain assumptions, reliable a posteriori error estimator for the fractional optimal control problem \eqref{def:minJ}--\eqref{eq:control_constraints}. As it was highlighted before, there is undoubtedly need for developing such an estimator and this is the first work that addresses this question for problem \eqref{def:minJ}--\eqref{eq:control_constraints}. Given a mesh $\T$ and corresponding approximations $\oue_{\T}$, $\ope_{\T}$ and $\ozsf_{\T}$, the proposed error indicator is built on the basis of three contributions: 
\[
 \E_{\textrm{ocp}} = \E_{\ue} + \E_{\pe} + \E_{\zsf},
\]
where $\E_{\ue}$ and $\E_{\pe}$ correspond to the a anisotropic posteriori error estimator on weighted Sobolev spaces of \cite{CNOS2}, for the state and adjoint equations, respectively. The error indicator $\E_{\zsf}$ is defined as the $\ell^2$-sum of the local contributions $\E_{\zsf}(\ozsf_{\T},\ope_{\T}; T) = \| \ozsf_{\T} - \Pi(-\tfrac{1}{\mu} \ope_{\T}(\cdot,0))\|_{L^2(\Omega)}$, with $T \in \T$ and $\Pi(v) = \min \{ \bsf, \max \{\asf,v\} \}$. We present an analysis for $\E_{\textrm{ocp}}$, we prove its efficiency and, under certain assumptions, its reliability. We remark that the devised error estimator is able to deal with both: the natural anisotropy of the mesh $\T$ in the extended variable and the degenerate coefficient $y^{\alpha}$. This approach is of value not only for the fractional optimal control problem, but in general for control problem involving anisotropic meshes since rigorous anisotropic a posteriori error estimators are scarce in the literature.
\section{Notation and preliminaries}
\label{sec:Prelim}
Throughout this work $\Omega$ is an open and bounded polytopal domain of $\R^n$ ($n\geq1$) with Lipschitz boundary $\partial\Omega$. We define the semi-infinite cylinder with base $\Omega$ and its lateral boundary, respectively, by
$
\C = \Omega \times (0,\infty)
$
and 
$
\partial_L \C  = \partial \Omega \times [0,\infty).
$
Given $\Y>0$, we define the truncated cylinder 
$
  \C_\Y = \Omega \times (0,\Y)
$
and $\partial_L\C_\Y$ accordingly. If $x\in \R^{n+1}$, we write 
$
  x =  (x^1,\ldots,x^n, x^{n+1}) = (x', x^{n+1}) = (x',y),
$
with $x^i \in \R$ for $i=1,\ldots,{n+1}$, $x' \in \R^n$ and $y\in\R$; this notation distinguishes the extended dimension $y$.

We denote by $\Laps$, $s \in (0,1)$, a fractional power of Dirichlet Laplace operator $(-\Delta)$. The parameter $\alpha$ belongs to $(-1,1)$ and is related to the power $s$ of the fractional Laplacian $\Laps$ by the formula $\alpha = 1 -2s$.

Finally, the relation $a \lesssim b$ indicates that $a \leq Cb$, with a constant $C$ that does not depend on $a$ or $b$ nor the discretization parameters. The value of $C$ might change at each occurrence.

\subsection{The fractional Laplace operator}
\label{subsec:fractional_Laplacian}
We adopt the spectral definition for the fractional powers of the Dirichlet Laplace operator \cite{CDDS:11,NOS}. The operator $(-\Delta)^{-1}:L^2(\Omega) \rightarrow L^2(\Omega)$ that solves $-\Delta w = f$ in $\Omega$ and $w = 0$ on $\partial \Omega$, is compact, symmetric and positive, whence its spectrum $\{ \lambda_k^{-1} \}_{k \in \mathbb{N}}$ is discrete, real, positive and accumulates at zero. Moreover, the eigenfunctions 
\[
 \{ \varphi_k \}_{k\in \mathbb{N}}: \quad -\Delta \varphi_k = \lambda_k \varphi_k \textrm{ in } \Omega, \quad \varphi_k = 0 \textrm{ on } \Omega, \quad k \in \mathbb{N},
\]
form an orthonormal basis of $L^2(\Omega)$. Fractional powers of $(-\Delta)$ can be defined by
\begin{equation*}
  (-\Delta)^s w  := \sum_{k=1}^\infty \lambda_k^{s} w_k \varphi_k, \qquad w \in C_0^{\infty}(\Omega), \qquad s \in (0,1),
\end{equation*} 
where $w_k = \int_{\Omega} w \varphi_k $. By density we extend this definition to
\begin{equation*}
\label{def:Hs}
  \Hs = \left\{ w = \sum_{k=1}^\infty w_k \varphi_k: 
  \sum_{k=1}^{\infty} \lambda_k^s w_k^2 < \infty \right\} = [H^1_0(\Omega),L^2(\Omega)]_{1-s}; 
\end{equation*}
see \cite{NOS} for details. For $ s \in (0,1)$ we denote by $\Hsd$ the dual space of $\Hs$.

\subsection{The Caffarelli--Silvestre extension problem}
\label{subsec:weighted}
In this section we explore problem \eqref{eq:alpha_harm} and its relation with the nonlocal problem \eqref{eq:fractional}; we refer the reader to \cite{CS:07,CDDS:11,NOS,ST:10} for details. Since $\alpha \in (-1,1)$, problem \eqref{eq:alpha_harm} is nonuniformly elliptic and thus it requires to introduce weighted Lebesgue and Sobolev spaces for its description. Let $E$ be an open set in $\R^{n+1}$. We define $L^2(|y|^{\alpha},E)$ as the Lebesgue space for the measure $|y|^\alpha \diff x$. We also define the weighted Sobolev space $H^1(|y|^{\alpha},E) := \{ w \in L^2(|y|^{\alpha},E): | \nabla w | \in L^2(|y|^{\alpha},E) \}$, which we endow with the norm
\begin{equation}
\label{wH1norm}
\| w \|_{H^1(|y|^{\alpha},E)} =
\left(  \| w \|^2_{L^2(|y|^{\alpha},E)} + \| \nabla w \|^2_{L^2(|y|^{\alpha},E)} \right)^{\frac{1}{2}}.
\end{equation}
Since $\alpha = 1-2s \in (-1,1)$, the weight $|y|^\alpha$ belongs to the Muckenhoupt class $A_2(\R^{n+1})$ \cite{Javier,Turesson}. Consequently, $H^1(|y|^{\alpha},D)$ is Hilbert and $C^{\infty}(\D) \cap H^1(|y|^{\alpha},D)$ is dense in $H^1(|y|^{\alpha},D)$ (cf.~\cite[Proposition 2.1.2, Corollary 2.1.6]{Turesson} and \cite[Theorem~1]{GU}).

The natural space to seek for a weak solution to problem \eqref{eq:alpha_harm} is 
\begin{equation*}
  \label{HL10}
  \HL(y^{\alpha},\C) := \left\{ w \in H^1(y^\alpha,\C): w = 0 \textrm{ on } \partial_L \C\right\}.
\end{equation*}

We recall the following \emph{weighted Poincar\'e inequality} \cite[inequality (2.21)]{NOS}:
\begin{equation*}
\label{Poincare_ineq}
\| w \|_{L^2(y^{\alpha},\C)} \lesssim \| \nabla v \|_{L^2(y^{\alpha},\C)}
\quad \forall w \in \HL(y^{\alpha},\C).
\end{equation*}
This yields that the seminorm on $\HL(y^{\alpha},\C)$ is equivalent to \eqref{wH1norm}. For $w \in H^1(y^{\alpha},\C)$ $\tr w$ denotes its trace onto $\Omega \times \{ 0 \}$. We recall (\cite[Prop.~2.5]{NOS} and \cite[Prop.~2.1]{CDDS:11})
\begin{equation}
\label{Trace_estimate}
\tr \HL(y^{\alpha},\C) = \Hs,
\qquad
  \|\tr w\|_{\Hs} \leq C_{\tr} \| w \|_{\HLn(y^{\alpha},\C)}.
\end{equation}
We must mention that $C_{\tr} \leq d_s^{-\frac{1}{2}}$ \cite[section 2.3]{CNOS2}, where $d_s = 2^{\alpha}\Gamma(1-s)/\Gamma(s)$. This will be useful in the analysis of the proposed a posteriori error indicator.

We conclude this section with the fundamental result by Caffarelli and Silvestre \cite{CS:07,CDDS:11,ST:10}: If $\usf \in \Hs$ and $\ue \in \HL(y^{\alpha},\C)$ solve \eqref{eq:fractional} and \eqref{eq:alpha_harm}, respectively, then
\[
 d_s \Laps \usf = \partial_{\nu^{\alpha}} \ue = -\lim_{y \rightarrow 0^+} y^{\alpha} \ue_y,
\]
in the sense of distributions. Here, $s \in (0,1)$ and $\alpha = 1-2s \in (-1,1)$.

\section{A priori error estimates}
\label{sec:apriori_control}

In an effort to make this work self-contained, in this section we review the results of \cite{AO}, where an a priori error analysis for a fully discrete approximation of the fractional optimal control problem is investigated. This will also serve to make clear the limitations of this theory.

\subsection{The extended optimal control problem}
\label{subsec:extended}

We start by recalling an equivalent problem to \eqref{def:minJ}--\eqref{eq:control_constraints}: the \emph{extended optimal control problem}. The main advantage of this problem is its local nature and is based on the Cafarelli--Silvestre extension result. To describe it, we define the set of \emph{admissible controls} as 
\begin{equation}
 \label{def:Zad}
 \Zad= \{ \wsf \in L^2(\Omega): \asf(x') \leq \wsf(x') \leq \bsf(x') \textrm{~~a.e~~}  x' \in \Omega \},
\end{equation}
where $\asf,\bsf \in L^2(\Omega)$ and satisfy the property $\asf(x') \leq \bsf(x')$ a.e. $x' \in \Omega$. The extended optimal control problem problem is then defined as follows: Find
$
 \text{min } J(\tr \ue,\zsf),
$
subject to the \emph{linear} state equation
\begin{equation}
\label{eq:alpha_harm_weak}
  a(\ue,\phi) = \langle \zsf, \tr \phi \rangle
    \quad \forall \phi \in \HL(y^{\alpha},\C),
\end{equation}
and the control constraints 
$
 \zsf \in \Zad.
$
The functional $J$ is defined by \eqref{def:minJ} with $\usf_d \in L^2(\Omega)$ and $\mu >0$. For $w, \phi \in \HL(y^{\alpha},\C)$, the bilinear form $a$ is defined by
\begin{equation*}
a(w,\phi) =  \frac{1}{d_s} \int_{\C} y^{\alpha} \nabla w \cdot \nabla \phi
\end{equation*}
and $\langle \cdot, \cdot \rangle$ denotes the duality pairing between $\Hs$ and $\Hsd$ which, as a consequence of \eqref{Trace_estimate}, is well defined for $\zsf \in \Hsd$ and $\phi \in \HL(y^{\alpha},\C)$.

The extended optimal control problem has a unique optimal solution $(\oue,\ozsf) \in \HL(y^{\alpha},\C) \times \Hs$ \cite[Theorem 3.11]{AO} and is equivalent to the fractional optimal control problem: $\tr \oue = \ousf$ \cite[Theorem 3.12]{AO}. 

\subsection{The truncated optimal control problem}
\label{subsec:truncated}

Since $\C$ is unbounded, problem \eqref{eq:alpha_harm_weak} cannot be directly approximated with finite-element-like techniques. However, as \cite[Proposition 3.1]{NOS} shows, the solution $\ue$ of problem \eqref{eq:alpha_harm_weak} decays exponentially in the extended variable $y$. This suggests to consider a \emph{truncated optimal control problem,} which is based on a truncation of the state equation \eqref{eq:alpha_harm_weak}. To describe it, we define
\[
  \HL(y^{\alpha},\C_\Y) = \left\{ w \in H^1(y^\alpha,\C_\Y): w = 0 \text{ on }
    \partial_L \C_\Y \cup \Omega \times \{ \Y\} \right\},
\]
and for all $w,\phi \in \HL(y^{\alpha},\C_\Y)$, the bilinear form
\begin{equation}
\label{def:a_Y}
a_\Y(w,\phi) = \frac{1}{d_s} \int_{\C_\Y} y^{\alpha}  \nabla w \cdot \nabla \phi.
\end{equation}
The truncated optimal control problem is then defined as follows: Find
$
 \text{min } J(\tr v,\rsf)
$
subject to the truncated state equation
\begin{equation}
\label{eq:alpha_harm_truncated}
  a_\Y(v,\phi) = \langle \rsf, \tr \phi \rangle
    \quad \forall \phi \in \HL(y^{\alpha},\C_\Y)
\end{equation}
and the control constraints 
$
 \rsf \in \Zad.
$
The existence and uniqueness of an optimal pair $(\bar v, \orsf) \in \HL(y^{\alpha},\C_{\T}) \times \Hs$ follows from \cite[Theorem 4.5]{AO}. In addition, we have that the optimal control $\orsf \in \Zad$ verifies the variational inequality
\begin{equation}
\label{eq:VI}
 (\tr \bar{p} + \mu \orsf , \rsf - \orsf )_{L^2(\Omega)} \geq 0 \quad \forall \rsf \in \Zad,
\end{equation}
where $\bar{p} \in \HL(y^{\alpha},\C_\Y)$ denotes the optimal adjoint state and solves
\begin{equation}
\label{eq:p_truncated}
a_\Y(\bar{p},\phi) = (  \tr \bar{v} - \usfd, \tr \phi )_{L^2(\Omega)}.  
\end{equation}

The following approximation properties follow from \cite[Lemma 4.6]{AO}: If $(\oue,\ozsf) \in \HL(y^{\alpha},\C) \times \Hs$ and $(\bar v, \orsf) \in \HL(y^{\alpha},\C_{\Y}) \times \Hs$ solve the extended and truncated optimal control problems, respectively, then
\begin{align*}
  \| \bar \zsf - \bar \rsf \|_{L^2(\Omega)} & \lesssim e^{-\sqrt{\lambda_1} \Y/4} \left(\| \bar{\rsf} \|_{L^2(\Omega)} + \| \usf_d \|_{L^2(\Omega)} \right),\\
  \| \nabla \left( \ue  - \bar{v}  \right) \|_{L^2(y^{\alpha},\C)} & \lesssim e^{-\sqrt{\lambda_1} \Y/4} \left(\| \bar{\rsf} \|_{L^2(\Omega)} + \| \usf_d \|_{L^2(\Omega)} \right),
\end{align*}
where $\lambda_1$ denotes the first eigenvalue of the operator $-\Delta$.

\subsection{A fully discrete scheme}
\label{subsec:fully}

In this section we recall the fully discrete scheme, proposed in \cite[section 5.3]{AO}, that approximates the solution to \eqref{def:minJ}--\eqref{eq:control_constraints}. We also review its a priori error analysis; see \cite[section 5.3]{AO} for details. To do so in this section, and this section only, we will assume the following regularity result, which is valid if, for instance, the domain $\Omega$ is convex \cite{Grisvard}:
\begin{equation}
\label{eq:Omega_regular}
 \| w \|_{H^2(\Omega)} \lesssim \| \Delta_{x'} w \|_{L^2(\Omega)} \quad \forall w \in H^2(\Omega) \cap H^1_0(\Omega). 
\end{equation}

The analysis of the fully discrete scheme of \cite[section 5.3]{AO} relies on the regularity properties of the optimal pairs $(\oue,\ozsf)$ and $(\bar v, \orsf)$ that solve the extended and truncated optimal control problems, respectively. We review such regularity properties in what follows. The results of \cite[Theorem 2.7]{NOS} reveals that the second order regularity of $\ue$, solving \eqref{eq:alpha_harm_weak}, is much worse in the extended direction, namely
\begin{align}
    \label{eq:reginx}
  \| \Delta_{x'} \ue\|_{L^2(y^{\alpha},\C)} + 
  \| \partial_y \nabla_{x'} \ue \|_{L^2(y^{\alpha},\C)}
  & \lesssim \| \zsf \|_{\Ws}, \\
\label{eq:reginy}
  \| \ue_{yy} \|_{L^2(y^{\beta},\C)} &\lesssim \| \zsf \|_{L^2(\Omega)},
\end{align}
where $\beta > 2\alpha + 1$. These result are also valid for the solution $v$ of problem \eqref{eq:alpha_harm_truncated}; see \cite[Remark 4.4]{NOS3}.

The estimates \eqref{eq:reginx}--\eqref{eq:reginy} have important consequences in the design of efficient numerical techniques to solve \eqref{eq:alpha_harm_weak}; they suggest that a graded mesh in the extended $(n+1)$--dimension must be used \cite[section 5]{NOS}. We recall the construction of the mesh over $\C_{\Y}$ used in \cite{AO,NOS}. First, we consider a graded partition $\mathcal{I}_{\Y}$ of the interval $[0,\Y]$ with mesh points
\begin{equation}
\label{eq:graded_mesh}
  y_k = \left( \frac{k}{M}\right)^{\gamma} \Y, \quad k=0,\dots,M,
\end{equation}
and $\gamma > 3/(1-\alpha)=3/(2s) > 1$. Second, we consider $\T_{\Omega} = \{ K \}$ to be a conforming mesh of $\Omega$, where $K \subset \R^n$ is an element that is isoparametrically equivalent either to the unit cube $[0,1]^n$ or the unit simplex in $\R^n$. We denote by $\Tr_{\Omega}$ the collections of all conforming refinements of an original mesh $\T_{\Omega}^0$. We assume that $\Tr_{\Omega}$ is shape regular \cite{CiarletBook}. We then construct a mesh $\T_{\Y}$ over $\C_{\Y}$ as the tensor product triangulation of $\T_{\Omega} \in \Tr_{\Omega}$ and $\mathcal{I}_{\Y}$. We denote by $\Tr$ the set of all the meshes obtained with this procedure, and recall that $\Tr$ satisfies the following weak shape regularity condition: If $T_1 = K_1 \times I_1$ and $T_2=K_2\times I_2 \in \T_\Y$ have nonempty intersection, then there exists a positive constant $\sigma_{\Y}$ such that
\begin{equation}
\label{eq:weak_shape_reg}
     h_{I_1} h_{I_2}^{-1} \leq \sigma_{\Y},
\end{equation}
where $h_I = |I|$. This weak shape regularity condition allows for anisotropy in the extended variable $y$ \cite{DL:05,NOS,NOS2}.

For $\T_{\Y} \in \Tr$, we define the finite element space 
\begin{equation}
\label{eq:FESpace}
  \V(\T_\Y) = \left\{
            W \in C^0( \overline{\C_\Y} ): W|_T \in \mathcal{P}_1(K) \otimes \mathbb{P}_1(I) \ \forall T \in \T_\Y, \
            W|_{\Gamma_D} = 0
          \right\},
\end{equation}
where $\Gamma_D = \partial_L \C_{\Y} \cup \Omega \times \{ \Y\}$ is the Dirichlet boundary. The space $\mathcal{P}_1(K)$ is $\mathbb{P}_1(K)$ -- the space of polynomials of degree at most $1$, when the base $K$ of $T = K \times I$ is a simplex. If $K$ is a cube, $\mathcal{P}_1(K)$ stand for $\mathbb{Q}_1(K)$ -- the space of polynomials of degree not larger that $1$ in each variable. We also define the space $\U(\T_{\Omega})=\tr \V(\T_{\Y})$, which is simply a $\mathcal{P}_1$ finite element space over the mesh $\T_\Omega$.

Before describing the numerical scheme introduced and developed in \cite{AO}, we recall the regularity properties of the extended and truncated optimal controls $\ozsf$ and $\orsf$, respectively. If $\usf_{d} \in \mathbb{H}^{1-s}(\Omega)$ and $\asf \leq 0 \leq \bsf$ for $s \in (0,\tfrac{1}{2}]$, then $\ozsf \in H^1(\Omega) \cap \mathbb{H}^{1-s}(\Omega)$ \cite[Lemmas 3.5 and 5.9]{AO}. Under the same framework, we have the same result for the truncated optimal control: $\orsf \in H^1(\Omega) \cap \mathbb{H}^{1-s}(\Omega)$ \cite[Lemma 5.9]{AO}.

After all these preparations, we are ready to describe the fully discrete scheme to approximate the fractional optimal control problem. The \emph{fully discrete optimal control problem} reads as follows: 
$
\min J(\tr V , Z),  
$
subject to the discrete state equation
\begin{equation}
\label{def:a_discrete}
a_\Y(V,W) =  ( Z, \tr W )_{L^2(\Omega)} \quad \forall W \in \V(\T_{\Y}),
\end{equation}
and the discrete control constraints
$
Z \in \mathbb{Z}_{ad}(\T_{\Omega}).
$
We recall that the functional $J$, the bilinear form $a_{\Y}$ and the discrete space $\V(\T_{\Y})$ are defined by \eqref{def:J}, \eqref{def:a_Y}, and \eqref{eq:FESpace}, respectively. The discrete and admissible set of controls is defined by
\begin{equation*}
\mathbb{Z}_{ad}(\T_{\Omega}) = \Zad \cap \left\{
            Z \in L^{\infty}( \Omega ): Z|_K \in \mathbb{P}_0(K) \quad \forall K \in \T_\Omega \right\},
\end{equation*}
\ie the space of piecewise constant functions defined on the partition $\T_{\Omega}$ that verifies the control bounds, which we assume to be real constants.

The existence and uniqueness of an optimal pair $(\bar{V}, \bar{Z}) \in \V(\T_\Y) \times \mathbb{Z}_{ad}(\T_{\Omega})$ solving the aforementioned problem is standard \cite[Theorem 5.15]{AO}. In addition, the optimal control $\bar Z \in \mathbb{Z}_{ad}(\T_{\Omega})$ is uniquely characterized by the variational inequality
\begin{equation}
\label{eq:VI_discrete}
(\tr \bar{P} + \mu \bar{Z}, Z- \bar{Z})_{L^2(\Omega)} \geq 0 \quad \forall Z \in \mathbb{Z}_{ad}(\T_{\Omega}),
\end{equation}
where the optimal and discrete adjoint state $\bar{P} \in \V(\T_\Y)$ solves
\begin{equation}
\label{eq:P_discrete}
a_\Y( \bar{P},W) = ( \tr \bar{V} - \usfd , \textrm{tr}_{\Omega} W )_{L^2(\Omega)} \quad \forall W \in \V(\T_{\Y}).
\end{equation}

With the discrete solution $\bar{V} \in \V(\T_\Y)$ at hand, we define
\begin{equation}
\label{eq:U_discrete}
\bar{U}:= \tr \bar{V},
\end{equation} 
and thus obtain a fully discrete approximation $(\bar{U},\bar{Z}) \in  \U(\T_{\Omega}) \times \mathbb{Z}_{ad}(\T_{\Omega})$ of the optimal pair $(\ousf,\ozsf) \in \Hs \times \Zad$ solving the fractional optimal control problem.

To write the a priori error estimates for the fully discrete optimal control problem, we notice that $\#\T_{\Y} = M \, \# \T_\Omega$, and that $\# \T_\Omega \approx M^n$ implies $\#\T_\Y \approx M^{n+1}$. Consequently, if $\T_\Omega$ is quasi-uniform, we have that $h_{\T_{\Omega}} \approx (\# \T_{\Omega})^{-1/n}$. We then have the following result \cite[Corollary 5.17]{AO}.

\begin{theorem}[fractional control problem: error estimate]
\label{th:error_estimate}
Let $(\bar{V},\bar{Z}) $ $\in \V(\T_\Y) \times \Zad$ solves the fully discrete control problem and $\bar{U} \in \U(\T_{\Omega})$ be defined as in \eqref{eq:U_discrete}. If $\Omega$ verifies \eqref{eq:Omega_regular}, $\usf_d \in \Ws$, and $\asf \leq 0 \leq \bsf$ for $s \in (0,\tfrac{1}{2}]$, then we have 
\begin{equation}
\label{fd2}
  \| \ozsf - \bar{Z} \|_{L^2(\Omega)} \lesssim  |\log (\# \T_{\Y})|^{2s}(\# \T_{\Y})^{\frac{-1}{n+1}} 
\left( \| \orsf \|_{H^1(\Omega)} + \| \usfd \|_{\Ws} \right),
\end{equation}
and
\begin{equation}
\label{fd1}
\| \ousf - \bar{U} \|_{\Hs} \lesssim  |\log (\# \T_{\Y})|^{2s}(\# \T_{\Y})^{\frac{-1}{n+1}} 
\left( \| \orsf \|_{H^1(\Omega)} + \| \usfd \|_{\Ws}  \right),
\end{equation} 
where the truncation parameter $\Y$, in the truncated optimal control problem, is chosen such that $\Y \approx \log( \# \T_{\Y} )$.
\end{theorem}

\begin{remark}[Domain and data regularity]\rm
The results of Theorem \ref{th:error_estimate} are valid if and only if $\Omega$ is such that \eqref{eq:Omega_regular} holds, $\usf_d \in \Ws$, and $\asf \leq 0 \leq \bsf$ for $s \in (0,\tfrac{1}{2}]$.
\end{remark}

\section{A posteriori error analysis}
The design and analysis of a posteriori error estimators for linear second-order elliptic boundary value problems on isotropic discretizations, \ie meshes where the aspect
ratio of all cells is bounded independently of the refinement level, has achieved a certain degree of maturity. Starting with the pioneering work of Babu{\v{s}}ka and Rheinboldt \cite{BR78}, a great deal of work has been devoted to its study. We refer the reader to \cite{AObook,BSbook,MNS02,NSV:09,NV,Verfurth} for an up-to-date discussion including also the design of AFEMs, their convergence and optimal complexity. In contrast to this well-established theory, the a posteriori error estimation on anisotropic discretizations, i.e., meshes where the cells have disparate sizes in each direction, is still not completely understood. To the best of our knowledge, the first work that introduces an a posteriori error estimator on anisotropic meshes is \cite{MR1389492}. The analysis provided in this work relies on certain assumptions on the mesh \cite[section 2]{MR1389492}, on the exact solution \cite[Definition 3.1]{MR1389492}, and on the discrete solution \cite[Definition 5.2]{MR1389492}. However, no explicit examples of AFEMs satisfying these assumptions are provided and their construction is not evident. Afterwards, the so--called \emph{matching function} is introduced in \cite{MR1785418,MR1777490} for deriving error indicators on anisotropic meshes. The presented analysis relies on the correct alignment of the grid with the exact solution. Indeed, the upper bound for the error involves the matching function, which depends on the error itself and then it does not provide a real computable quantity; see \cite[Theorem 2]{MR1785418} and see \cite[Theorem 5.1]{MR1777490}. The effect of approximating the matching function with a recovered gradient based technique is discussed in \cite{MR1785418,MR1777490}.

To the best of our knowledge, the first paper that attempts to deal with an anisotropic a posteriori error estimator for an optimal control problem is \cite{MR2257635}. In this work, the author proposes, based on the the goal--oriented approach developed in \cite{MR1780911}, an anisotropic error indicator for a parabolic optimal control problem involving the heat equation. However, the presented upper bound for the error \cite[Proposition 7]{MR2257635} depends on the exact solution and therefore, it is not computable; see the discussion in \cite[section 5]{MR2257635}. Later, reference \cite{MR2837575} presents an anisotropic posteriori error estimator for an optimal control problem of a scalar advection--reaction--diffusion equation. The analysis relies on the goal--oriented approach of \cite{MR1780911}, and the a priori and posteriori error analyses of \cite{MR1865506} and \cite{MR1971213}, respectively. The presented upper bound for the error depends on the exact optimal variables and therefore is not computable \cite[Proposition 3.5]{MR2837575}. This shortcoming is circumvented, computationally, by invoking a suitable recovery procedure. 

The main contribution of this work is the design and study of an a posteriori error indicator for the fractional optimal control problem \eqref{def:minJ}--\eqref{eq:control_constraints}. To accomplish this task, we invoke the a posteriori error indicator developed in \cite{CNOS2} that is based on the solution of local problems on stars; we remark that, since problems \eqref{eq:alpha_harm_truncated} and \eqref{eq:p_truncated} involve the coefficient $y^{\alpha}$ $(-1 < \alpha < 1)$, that is not uniformly bounded, the usual residual estimator does not apply. The idea of working on stars goes back to Babu{\v{s}}ka and Miller \cite{MR880421}, who introduced local Dirichlet problems. Later, references \cite{CF:00,MNS02} proposed solving local weighted problems on stars that deliver rather good effectivity indices. A convergence proof of AFEM driven by such error indicators is provided in \cite{MNS02} for a Poisson problem, and in \cite{MR2875241} for a general second-order elliptic PDE; the latter also includes optimal complexity.  We also refer the reader to \cite{10.2307/2007953} for estimators based on solving Neumann problems on elements and their further improvements via the so-called \emph{flux equilibration principle} \cite{AObook}.

Concerning the a posteriori error analysis for \eqref{def:minJ}--\eqref{eq:control_constraints}, we first propose and explore an ideal anisotropic error indicator that is constructed on the basis of solving local problems on \emph{cylindrical stars}. This indicator is able to deal with both: the coefficient $y^{\alpha}$ and the anisotropic mesh $\T_{\Y}$. Under a computationally implementable geometric condition imposed on the mesh, which does not depend on the exact optimal variables, we derive the equivalence between the ideal estimator and the error without oscillation terms. This ideal indicator sets the basis to define a computable error estimator, which, under certain assumptions, is equivalent to the error up to data oscillations terms. 

\subsection{Preliminaries}
\label{subsec:preliminaries}

Let us begin the discussion on a posteriori error estimation with some terminology and notation that follows from \cite{CNOS2}. Given a node $z$ on the mesh $\T_{\Y}$, we write $z = (z',z'')$ where $z'$ and $z''$ are nodes on the meshes $\T_{\Omega}$ and $\mathcal{I}_{\Y}$ respectively. 

Given $K \in \T_{\Omega}$, we denote by $\N(K)$ the set of nodes of $K$ and by $\Nin(K)$ the set of interior nodes. With this notation at hand, we define 
$
\N(\T_{\Omega}) = \cup \{ \N(K): K \in \T_\Omega \}
$ and
$
\Nin(\T_{\Omega}) = \cup \{\Nin(K): K \in \T_\Omega \}.
$
Given $T \in \T_{\Y}$, we define $\N(T)$, $\Nin(T)$, and then $\Nin(\T_{\Y})$ and $\N(\T_{\Y})$ accordingly. 

Given $z' \in \N(\T_{\Omega})$, we define the \emph{star} around $z'$ as
\[
  S_{z'} = \bigcup \left\{ K \in \T_\Omega : \ K \ni z' \right\} \subset \Omega 
\]
and the \emph{cylindrical star} around $z'$ as
\begin{equation}
\label{def:cylindrical_star}
  \C_{z'} := \bigcup\left\{ T \in \T_\Y : T = K \times I,\ K \ni z'  \right\}= S_{z'} \times (0,\Y) \subset \C_{\Y}.
\end{equation}

Given $K \in \T_{\Omega}$ we define its \emph{patch} as 
$
  S_K := \bigcup_{z' \in K} S_{z'}.
$
For $T \in \T_\Y$ its patch $S_T$ is defined similarly. Given $z' \in \N(\T_{\Omega})$
we define its \emph{cylindrical patch} as
\[
\D_{z'} := \bigcup  \left\{ \C_{w'}: w' \in S_{z'} \right\} \subset \C_{\Y}.
\]

For each $z' \in \N(\T_{\Omega})$ we set $h_{z'} := \min\{h_{K}: K \ni z' \}$. 
%

\subsection{Local weighted Sobolev spaces}
\label{subsec:local_spaces}

To define the a posteriori error estimator proposed in this work, we need to introduce some local weighted Sobolev spaces.

\begin{definition}[local spaces]
Given $z' \in \N(\T_\Omega)$, we define
\begin{equation}
\label{eq:local_space}
\W(\C_{z'}) = \left \{ w \in H^1(y^{\alpha},\C_{z'} ): w = 0 \textrm{ on } \partial \C_{z'} 
\setminus \Omega \times \{ 0\} \right \},
\end{equation}
where $\C_{z'}$ denotes the cylindrical star around $z'$  defined in \eqref{def:cylindrical_star}.
\end{definition}

Since $y^\alpha$ belongs to the class $A_2(\R^{n+1})$ \cite{Javier,Muckenhoupt}, the space $\W(\C_{z'})$ is Hilbert. In addition, we have the following weighted Poincar\'e-type inequality \cite[Proposition 5.8]{CNOS2}: If $w \in \W(\C_{z'})$, then
\begin{equation}
\label{eq:Poincare}
  \| w \|_{L^2(y^{\alpha},\C_{z'})} \lesssim \Y \|  \nabla w \|_{L^2(y^{\alpha},\C_{z'})},
\end{equation}
where $\Y$ denotes the truncation parameter introduced in section \ref{subsec:truncated}. We also have the following trace inequality that follows from \cite[Proposition 2.1]{CDDS:11}: If $w \in \W(\C_{z'})$, then 
\begin{equation}
 \label{Trace_estimate_local}
 \| \tr w \|_{L^2(S_{z'})} \leq C_{\tr} \| \nabla w \|_{L^2(y^{\alpha},\C_{z'})}.
\end{equation}
We notice that the same arguments of \cite[Section 2.3]{CNOS2} yield $C_{\tr} \leq d_s^{-\frac{1}{2}}$.

\subsection{An ideal a posteriori error estimator}
\label{subsec:ideal_a_posteriori}

On the basis of the notation introduced in subsections \ref{subsec:preliminaries} and \ref{subsec:local_spaces}, we propose and analyze an ideal a posteriori error estimator for the fractional optimal control problem \eqref{def:minJ}--\eqref{eq:control_constraints}. The proposed error indicator is ideal because it is not computable: it is based on the resolution of local problems on infinite dimensional spaces. However, it provides the intuition required to define a discrete and computable error indicator, as is explained in section \ref{subsec:computable_a_posteriori}. The construction of this ideal indicator allows for the anisotropic meshes $\T_{\Y}$ defined in section \ref{sec:apriori_control} and the nonuniformly coefficient $y^{\alpha}$ of problem \eqref{eq:alpha_harm_weak}. We prove that is equivalent to the error without oscillation terms. 

The ideal error indicator is defined as the sum of three contributions:
\begin{equation}
 \label{eq:defofEocp}
 \E_{\textrm{ocp}}(\bar{V},\bar{P},\bar{Z}; \T_{\Y}) = \E_{V}(\bar{V},\bar{Z}; \N(\T_{\Omega})) + \E_{P}(\bar{P},\bar{V}; \N(\T_{\Omega})) + \E_{Z}(\bar{Z},\bar{P}; \T_{\Omega}),
\end{equation}
where $\T_{\Y} \in \Tr$ corresponds to the anisotropic mesh constructed in subsection \ref{subsec:fully} and $\bar{V}$, $\bar{P}$ and $\bar{Z}$ denote the optimal variables solving the fully discrete optimal control problem described in subsection \ref{subsec:fully}. We now proceed to describe each contribution in \eqref{eq:defofEocp} separately.
To do this, we introduce, for $w,\psi \in \W(\C_{z'})$, the bilinear form
\begin{equation}
 \label{eq:a_local}
 a_{z'}(w,\psi) = \frac{1}{d_s} \int_{\C_{z'}} y^{\alpha} \nabla w  \nabla \psi.
\end{equation}
Then, the first contribution in \eqref{eq:defofEocp} is defined on the basis of the indicator developed in \cite[section 5.3]{CNOS2}. We define $\zeta_{z'} \in \W(\C_{z'})$ as the solution to 
\begin{equation}
\label{eq:ideal_local_problemV}
a_{z'}(\zeta_{z'},\psi) = \langle \bar{Z}, \tr \psi  \rangle  -   a_{z'}(\bar{V}, \psi) \quad \forall \psi \in \W(\C_{z'}),
\end{equation}
where we recall that the space $\W(\C_{z'})$ is defined in \eqref{eq:local_space}. With this definition at hand, we then define the local error estimator
\begin{equation}
\label{eq:defofEV} 
\E_{V}(\bar{V},\bar{Z}; \C_{z'}) := \| \nabla\zeta_{z'} \|_{L^2(y^{\alpha},\C_{z'})}
\end{equation}
and the global error estimator
$
\E_{V}(\bar{V},\bar{Z}; \N(\T_{\Omega})) := \left( \sum_{z' \in \N(\T_{\Omega})} \E_{V}^2(\bar{V},\bar{Z}; \C_{z'}) \right)^{\frac{1}{2}}.
$

We now describe the second contribution in \eqref{eq:defofEocp}. To accomplish this task, we define $\chi_{z'} \in \W(\C_{z'})$ as the solution to the local problem
\begin{equation}
\label{eq:ideal_local_problemP}
a_{z'}(\chi_{z'},\psi) = \langle \tr \bar{V} - \usf_d, \tr \psi  \rangle  - a_{z'}(\bar{P}, \psi) \quad \forall \psi \in \W(\C_{z'}).
\end{equation}
We then define the local error indicator 
\begin{equation}
\label{eq:defofEP} 
\E_{P}(\bar{P},\bar{V}; \C_{z'}) := \| \nabla\chi_{z'} \|_{L^2(y^{\alpha},\C_{z'})}
\end{equation}
and the global error indicator $\E_{P}(\bar{P},\bar{V}; \N( \T_{\Omega} ) ) := \left( \sum_{z' \in \N( \T_{\Omega} )} \E_{P}^2(\bar{P},\bar{V}; \C_{z'}) \right)^2$.

Finally, we define a global error estimator for the optimal control as follows:
\begin{equation}
\label{eq:defofEZglobal}
  \E_{Z}(\bar{Z},\bar{P}; \T_{\Omega}) := \left( \sum_{K \in \T_{\Omega}}  \E^2_{Z}(\bar{Z},\bar{P}; K) \right)^{1/2},
\end{equation}
with the local error indicators
\begin{equation}
\label{eq:defofEZ}
\E_{Z}(\bar{Z},\bar{P}; K) := \| \bar{Z} - \Pi (-\tfrac{1}{\mu} \tr \bar{P}) \|_{L^2(K)}.
\end{equation}
In \eqref{eq:defofEZ}, $\Pi: L^2(\Omega) \rightarrow \Zad$ denotes the nonlinear projection operator defined by 
\begin{equation} 
\label{def:Pi}
  \Pi(x') = \min\{\bsf, \max\{\asf, x'\}\},
\end{equation}
where $\asf$ and $\bsf$ denote the control bounds defining the set $\Zad$ in \eqref{def:Zad}.

To invoke the results of \cite[section 5.3]{CNOS2}, we introduce an implementable geometric condition that will allow us to consider graded meshes in $\Omega$ while preserving the anisotropy in the extended direction $y$ that is necessary to retain optimal orders of approximation. The flexibility of having graded meshes in $\Omega$ is essential for compensating some possible singularities in the $x'$--variables. We thus assume the following condition over the family of triangulations $\Tr$: there exists a positive constant $C_{\Tr}$ such that, for every mesh $\T_{\Y} \in \Tr$, we have that
\begin{equation}
\label{condition}
 h_{\Y} \leq C_{\Tr} h_{z'},
\end{equation}
for all interior nodes $z'$ of $\T_{\Omega}$. Here, $h_{\Y}$ denotes the largest size in the $y$--direction. We remark that this condition is fully implementable.

We now derive an estimate of the energy error in terms of the total
error estimator $\E_{\textrm{ocp}}$ defined in \eqref{eq:defofEocp} (reliability).

\begin{theorem}[global upper bound]
Let $(\bar{v}, \bar{p}, \orsf) \in \HL(y^{\alpha},\C_{\Y}) \times \HL(y^{\alpha},\C_{\Y}) \times \Zad$ be the solution to the optimality system associated with the truncated optimal control problem defined in subsection \ref{subsec:truncated} and $(\bar{V},\bar{P},\bar{Z}) \in \V(\T_{\Y}) \times \V(\T_{\Y}) \times \mathbb{Z}_{ad}(\T_{\Omega})$ its numerical approximation defined in subsection \ref{subsec:fully}. If \eqref{condition} holds, then
\begin{multline}
\label{eq:reliability}
 \| \nabla( \bar{v} -\bar{V}) \|_{L^2(y^{\alpha},\C_{\Y})}  +  \| \nabla( \bar{p} -\bar{P}) \|_{L^2(y^{\alpha},\C_{\Y})} + \| \bar{\rsf} -\bar{Z} \|_{L^2(\Omega)} \\ \lesssim \E_{V}(\bar{V},\bar{Z}; \N(\T_{\Omega})) + \E_{P}(\bar{P},\bar{V}; \N(\T_{\Omega})) +  \E_{Z}(\bar{Z},\bar{P}; \T_{\Omega}),
\end{multline}
where the hidden constant is independent of the continuous and discrete optimal variables, and the size of the elements in the meshes $\T_{\Omega}$ and $\T_{\Y}$.
\label{th:reliability}
\end{theorem}
\begin{proof}
The proof involves six steps.
 
Step 1. With the definition \eqref{eq:defofEZ} of the local error indicator $\E_{Z}$ in mind, we define the auxiliary control $\tilde{\rsf} = \Pi (-\frac{1}{\mu} \tr \bar{P})$ and notice that it verifies 
\begin{equation}
\label{eq:VI_rtilde}
 ( \tr \bar{P} + \mu \tilde{\rsf} , \rsf - \tilde{\rsf} )_{L^2(\Omega)} \geq 0 \quad \forall \rsf \in \Zad.
\end{equation}
Then, an application of the triangle inequality yields
\begin{equation}
\label{eq:r-Z-rtilde}
 \| \orsf - \bar{Z} \|_{L^2(\Omega)} \leq  \| \orsf - \tilde{\rsf} \|_{L^2(\Omega)}  +  \| \tilde{\rsf}- \bar{Z} \|_{L^2(\Omega)}
\end{equation}
We notice that the second term on the right hand side of the previous inequality corresponds to the definition of the global indicator \eqref{eq:defofEZglobal}. Thus, it suffices to bound the first term, \ie $\| \orsf - \tilde{\rsf} \|_{L^2(\Omega)}$.

Step 2. Set $\rsf = \tilde \rsf$ in \eqref{eq:VI} and $\rsf = \orsf$ in \eqref{eq:VI_rtilde}. Adding the obtained inequalities we arrive at
\begin{equation}
\label{eq:Step_2}
\mu \| \orsf - \tilde{\rsf} \|^2_{L^2(\Omega)} \leq (\tr(\bar{p} - \bar{P}), \tilde{\rsf} - \orsf )_{L^2(\Omega)},
\end{equation}
where $\bar{p}$ and $\bar{P}$ solve \eqref{eq:p_truncated} and \eqref{eq:P_discrete}, respectively. To control the right hand side of this expression, we introduce the auxiliary adjoint state $q$ that uniquely solves
\begin{equation}
\label{eq:q_truncated}
q \in \HL(y^{\alpha},\C_{\Y}):\quad a_\Y(\phi,q) = (  \tr \bar{V} - \usfd, \tr \phi )_{L^2(\Omega)} \quad \forall \phi \in \HL(y^{\alpha},\C_{\Y}).
\end{equation}
By writing $\bar{p} - \bar{P} = (\bar{p}-q) + (q - \bar{P})$, the estimate \eqref{eq:Step_2} immediately yields
\begin{equation}
\label{eq:muleq}
\mu \| \orsf - \tilde{\rsf} \|^2_{L^2(\Omega)} \leq  (\tr(\bar{p} - q), \tilde{\rsf} - \orsf )_{L^2(\Omega)} + (\tr(q- \bar{P}), \tilde{\rsf} - \orsf )_{L^2(\Omega)}.
\end{equation}
We conclude this step by noticing that, by construction, problem \eqref{eq:P_discrete} corresponds to the Galerkin approximation of \eqref{eq:q_truncated}. Then, \cite[Proposition 5.14]{CNOS2} yields
\begin{align}
\nonumber
  |\textrm{II}| & := |(\tr(q- \bar{P}), \tilde{\rsf} - \orsf )_{L^2(\Omega)}| \lesssim \| \nabla( q- \bar{P} ) \|_{L^2(y^{\alpha},\C_{\Y})} \| \tilde{\rsf} - \orsf \|_{L^2(\Omega)} \\ & \lesssim \E_{P}(\bar{P},\bar{V};\N(\T_{\Omega}))  \| \tilde{\rsf} - \orsf \|_{L^2(\Omega)} \leq \frac{\mu}{4}  \| \tilde{\rsf} - \orsf \|^2_{L^2(\Omega)} + C \E_{P}(\bar{P},\bar{V};\N(\T_{\Omega})),
\label{eq:II}
\end{align}
where in the first inequality we used \eqref{Trace_estimate}; $C$ denotes a positive constant.

Step 3. The goal of this step is to bound the term $\mathrm{I}:= (\tr(\bar{p} - q), \tilde{\rsf} - \orsf )_{L^2(\Omega)}$. To accomplish this task, we introduce another auxiliary adjoint state
\begin{equation}
 \label{eq:w_truncated}
w \in \HL(y^{\alpha},\C_{\Y}): \quad a_{\Y} (\phi,w) = (  \tr \tilde{v} - \usfd, \tr \phi )_{L^2(\Omega)} \quad \forall \phi \in \HL(y^{\alpha},\C_{\Y}),
\end{equation}
where $\tilde v$ is defined as the unique solution to
\begin{equation}
\label{def:tildev}
\tilde{v} \in \HL(y^{\alpha},\C_{\Y}): \quad a_{\Y} (\tilde{v},\phi) = ( \tilde{\rsf}, \tr \phi )_{L^2(\Omega)} \quad \forall \phi \in \HL(y^{\alpha},\C_{\Y}),
\end{equation}
and $\tilde{\rsf} = \Pi (-\frac{1}{\mu} \tr \bar{P})$. We then write $\bar{p} - q = ( \bar{p} - w ) + (w - q )$ and bound each contribution to the term $\mathrm{I}$ separately. To do this, we observe that $\bar{v} - \tilde{v}$ solves the problem $a_{\Y} (\bar{v} - \tilde{v},\phi) = (\orsf- \tilde{\rsf}, \tr \phi )_{L^2(\Omega)}$ for all $\phi \in \HL(y^{\alpha},\C_{\Y})$. On the other hand, for all these test functions, $\bar{p} - w$ solves $a_{\Y} (\phi,\bar{p} - w) = (  \tr (\bar{v}-\tilde{v}), \tr \phi )_{L^2(\Omega)}$. Combining these two problems, we arrive at
\begin{equation}
\label{eq:I_1}
 \mathrm{I}_1:= (\tr(\bar{p} - w), \tilde{\rsf} - \orsf )_{L^2(\Omega)} = - a_{\Y}  (\bar{v} - \tilde{v},\bar p -w) = - \| \tr (\bar{v} - \tilde{v}) \|^2_{L^2(\Omega)} \leq 0.
\end{equation}

We now estimate the term $\mathrm{I}_2:= (\tr(w - q), \tilde{\rsf} - \orsf )_{L^2(\Omega)}$, where $w$ and $q$ solve problems \eqref{eq:w_truncated} and \eqref{eq:q_truncated}, respectively. We observe that the difference $w - q$ solves $a_{\Y} (\phi,w-q) = (  \tr (\tilde{v}-\bar{V}), \tr \phi )_{L^2(\Omega)}$ for all $\phi \in \HL(y^{\alpha},\C_{\Y})$. Thus, the trace estimate \eqref{Trace_estimate} and the stability of problem \eqref{eq:q_truncated} yield
\begin{equation}
\label{eq:aux1}
 |\mathrm{I}_2| \lesssim \| \nabla ( w-q )  \|_{L^2(y^{\alpha},\C_{\Y})} \|\tilde{\rsf} - \orsf  \|_{L^2(\Omega)} \lesssim \|  \tr (\tilde{v}-\bar{V})  \|_{L^2(\Omega)} \|\tilde{\rsf} - \orsf  \|_{L^2(\Omega)}.
\end{equation}
It suffices to bound the term $\|  \tr (\tilde{v}-\bar{V})  \|_{L^2(\Omega)}$. To accomplish this task, we invoke the triangle inequality and obtain the estimate $\|  \tr (\tilde{v}-\bar{V})  \|_{L^2(\Omega)} \leq \|  \tr (\tilde{v}-v^*)  \|_{L^2(\Omega)}+ \|  \tr (v^*-\bar{V})  \|_{L^2(\Omega)}$, where $v^*$ denotes the unique solution to the following problem:
\begin{equation}
 \label{eq:v_star}
v^* \in \HL(y^{\alpha},\C_{\Y}): \quad a_{\Y} (v^*,\phi) = (  \bar{Z}, \tr \phi )_{L^2(\Omega)} \quad \forall \phi \in \HL(y^{\alpha},\C_{\Y}).
\end{equation}
Now, we invoke \eqref{Trace_estimate} and the stability of \eqref{eq:v_star} to derive that $\|  \tr (\tilde{v}-v^*)  \|_{L^2(\Omega)} \lesssim \| \tilde \rsf - \bar{Z}\|_{L^2(\Omega)}$. This, in view of the definition of $\E_{Z}$, given by \eqref{eq:defofEZglobal}--\eqref{eq:defofEZ}, yields  
\begin{equation}
\label{eq:aux2}
\|  \tr (\tilde{v}-v^*)  \|_{L^2(\Omega)} \lesssim \E_{Z}(\bar{Z},\bar{P};\T_{\Omega}). 
\end{equation}
To control the remainder term, we observe that problem \eqref{def:a_discrete} corresponds to the Galerkin approximation of \eqref{eq:v_star}. Consequently, \eqref{Trace_estimate} and \cite[Proposition 5.14]{CNOS2} yield 
\begin{equation}
\label{eq:aux3}
\|  \tr (v^*-\bar{V})  \|_{L^2(\Omega)} \lesssim \| \nabla( v^*-\bar{V} ) \|_{L^2(y^{\alpha},\C_{\Y})} \lesssim \E_{V}(\bar{V},\bar{Z};\N(\T_{\Omega})). 
\end{equation}
In view of \eqref{eq:aux1}, the collection of the estimates \eqref{eq:aux2} and \eqref{eq:aux3} allows us to obtain 
\[
 |\mathrm{I}_2| \leq \frac{\mu}{4}\|\tilde{\rsf} - \orsf  \|^2_{L^2(\Omega)} + C \left( \E^2_{Z}(\bar{Z},\bar{P};\T_{\Omega}) +  \E^2_{V}(\bar{V},\bar{Z};\N(\T_{\Omega})) \right),
\]
where $C$ denotes a positive constant. Since \eqref{eq:I_1} tells us that $\mathrm{I}_1 \leq 0$, we obtain a similar estimate for the term $\mathrm{I} = \mathrm{I}_1 + \mathrm{I}_2$. This estimate implies, on the basis of \eqref{eq:muleq} and \eqref{eq:II}, the following bound
\[
\| \orsf - \tilde{\rsf} \|^2_{L^2(\Omega)} \lesssim \E^2_{V}(\bar{V},\bar{Z};\N(\T_{\Omega})) + \E^2_{P}(\bar{P},\bar{V};\T_{\Omega})+ \E^2_{Z}(\bar{Z},\bar{P};\T_{\Omega}),
\]
which, invoking \eqref{eq:r-Z-rtilde}, provides an estimate for the error in control approximation:
\begin{equation}
 \label{eq:control_error}
 \| \orsf - \bar{Z} \|_{L^2(\Omega)} \lesssim \E_{\textrm{ocp}}(\bar{V},\bar{P},\bar{Z}; \T_{\Y}).
\end{equation}

Step 4. The goal of this step in to bound the error $\| \nabla(\bar{v}-\bar{V}) \|_{L^2(y^{\alpha},\C_{\Y})}$ in terms of the ideal error indicator \eqref{eq:defofEocp}. We employ similar arguments to the ones developed in step 2. We write $\bar{v}-\bar{V} = (\bar{v}- v^*) + (v^*- \bar{V})$, where $v^*$ is defined in \eqref{eq:v_star}. The stability of problem \eqref{eq:v_star} and the estimate \eqref{eq:control_error} immediately provide the bound
\[
 \| \nabla(\bar{v}- v^*) \|_{L^2(y^{\alpha},\C_{\Y})} \lesssim \| \orsf - \bar{Z} \|_{L^2(\Omega)} \lesssim \E_{\textrm{ocp}}(\bar{V},\bar{P},\bar{Z}; \T_{\Y}).
\]
which, combined with \eqref{eq:aux3}, allows us to derive
\begin{equation}
\label{eq:state_error}
 \| \nabla(\bar{v}-\bar{V}) \|_{L^2(y^{\alpha},\C_{\Y})} \lesssim \E_{\textrm{ocp}}(\bar{V},\bar{P},\bar{Z}; \T_{\Y}).
 \end{equation}
 
 Step 5. We bound the term $\| \nabla(\bar{p}-\bar{P}) \|_{L^2(y^{\alpha},\C_{\Y})}$. To accomplish this task, we invoke the triangle inequality and write
 \[
  \| \nabla(\bar{p}-\bar{P}) \|_{L^2(y^{\alpha},\C_{\Y})} \leq \| \nabla(\bar{p}-q) \|_{L^2(y^{\alpha},\C_{\Y})} + \| \nabla(q-\bar{P}) \|_{L^2(y^{\alpha},\C_{\Y})}.
 \]
where $q$ is defined as the solution to problem \eqref{eq:q_truncated}. Applying the stability of problem \eqref{eq:q_truncated}, the trace estimate \eqref{Trace_estimate}, and \eqref{eq:state_error}, we arrive at
\[
  \| \nabla(\bar{p}-q) \|_{L^2(y^{\alpha},\C_{\Y})} \lesssim \| \tr(\bar v - \bar V)\|_{L^2(y^{\alpha},\C_{\Y})} \lesssim \E_{\textrm{ocp}}(\bar{V},\bar{P},\bar{Z}; \T_{\Y}).
\]
On the other hand, since $\bar P$, solution to \eqref{eq:P_discrete}, corresponds to the Galerkin approximation of $q$, solution to \eqref{eq:q_truncated}, we invoke
\cite[Proposition 5.14]{CNOS2} to derive
\[
 \| \nabla(q-\bar{P})  \|_{L^2(y^{\alpha},\C_{\Y})} \lesssim \E_{P}(\bar P, \bar Z; \N(\T_{\Omega})).
\]
Collecting the derived estimates, we obtain that
\begin{equation}
\label{eq:adjoint_error}
 \| \nabla(\bar{p}-\bar P) \|_{L^2(y^{\alpha},\C_{\Y})} \lesssim \E_{\textrm{ocp}}(\bar{V},\bar{P},\bar{Z}; \T_{\Y}). 
\end{equation}

Step 6. Finally, the desired estimate \eqref{eq:reliability} follows from a simple collection of the estimates \eqref{eq:control_error}, \eqref{eq:state_error} and \eqref{eq:adjoint_error}.
\end{proof}

We now derive a local lower bound that measures the quality of $\E_{\textrm{ocp}}$ (efficiency). To achieve this, we define 
\begin{equation}
\label{eq:C}
 C(d_s,\mu) = \max \{  2d_s^{-1}, d_s^{-\frac{1}{2}}(\mu^{-1} + d_s^{-\frac{1}{2}}),1 + d_s^{-\frac{1}{2}} \}.
\end{equation}

\begin{theorem}[local lower bound]
Let $(\bar{v}, \bar{p}, \orsf) \in \HL(y^{\alpha},\C_{\Y}) \times \HL(y^{\alpha},\C_{\Y}) \times \Zad$ be the solution to the optimality system associated with the truncated optimal control problem defined in subsection \ref{subsec:truncated} and $(\bar{V},\bar{P},\bar{Z}) \in \V(\T_{\Y}) \times \V(\T_{\Y}) \times \mathbb{Z}_{ad}(\T_{\Omega})$ its numerical approximation defined in subsection \ref{subsec:fully}. Then,
\begin{multline}
\label{eq:efficiency}
 \E_{V}(\bar{V},\bar{Z}; \C_{z'} ) + \E_{P}(\bar{P},\bar{V}; \C_{z'} ) +  \E_{Z}(\bar{Z},\bar{P}; S_{z'}) \\ 
 \leq C(s,\mu) \left( \| \nabla( \bar{v} -\bar{V}) \|_{L^2(y^{\alpha},\C_{z'} )}  +  \| \nabla( \bar{p} -\bar{P}) \|_{L^2(y^{\alpha},\C_{z'})} + \| \bar{\rsf} -\bar{Z} \|_{L^2(S_{z'})} \right),
\end{multline}
where $C(d_s,\mu)$ depends only on $d_s$ and the parameter $\mu$ and is defined in \eqref{eq:C}.
\label{th:global_efficiency}
\end{theorem}
\begin{proof} We proceed in three steps.

Step 1. We begin by analyzing the efficiency properties of the indicator $\E_{V}$ defined, locally, by \eqref{eq:defofEV}. Let $z' \in \N(\T_{\Omega})$. We invoke the fact that $\zeta_{z'}$ solves the local problem \eqref{eq:ideal_local_problemV} and conclude that
\[
\E_{V}^2(\bar V, \bar Z; \C_{z'}) = a_{z'}( \zeta_{z'}, \zeta_{z'}) = \langle \orsf, \tr\zeta_{z'} \rangle + \langle \bar Z - \orsf, \tr\zeta_{z'} \rangle - a_{z'}( \bar V, \zeta_{z'}).
\]
Define $e_V = \bar{v} - \bar{V}$, where $\bar v$ solves \eqref{eq:alpha_harm_truncated}. Invoking \eqref{Trace_estimate_local} with $C_{\tr} \leq d_s^{-\frac{1}{2}}$ and a simple application of the Cauchy-Schwarz inequality, we arrive at 
\begin{align*}
\E_{V}^2(\bar V, \bar Z; \C_{z'}) & \leq d_s^{-1}\|  \nabla e_{V} \|_{L^2(y^{\alpha},\C_{z'})}\| \nabla \zeta_{z'}\|_{L^2(y^{\alpha},\C_{z'})} + \| \orsf -\bar{Z} \|_{L^2(S_{z'})} \| \tr \zeta_{z'}\|_{L^2(S_{z'})}
 \\
 & \leq \left( d_s^{-1}\|  \nabla e_{V} \|_{L^2(y^{\alpha},\C_{z'})} + d_s^{-\frac{1}{2}}\| \orsf -\bar{Z} \|_{L^2(S_{z'})} \right) \|  \nabla \zeta_{z'}\|_{L^2(y^{\alpha},\C_{z'})}.
\end{align*}
This, in view of definition \eqref{eq:defofEV}, implies the efficiency of $\E_{V}$:
\begin{equation}
\label{eq:efficiency_V}
\E_{V} (\bar V, \bar Z; \C_{z'}) \leq d_s^{-1} \|  \nabla e_{V} \|_{L^2(y^{\alpha},\C_{z'})} + d_s^{-\frac{1}{2}} \| \orsf -\bar{Z} \|_{L^2(S_{z'})}.
\end{equation}

Step 2. In this step we elucidate the efficiency properties of the indicator $\E_{P}$ defined in \eqref{eq:defofEP}. Following the arguments elaborated in step 1, we write 
\[
 \E_{P}^2(\bar P, \bar V; \C_{z'}) = \langle \tr(\bar V - \bar v), \chi_{z'}  \rangle + a_{z'}( e_{P}, \chi_{z'}),
\]
where $\chi_{z'} \in \W(\C_{z'})$ solves \eqref{eq:ideal_local_problemP}. An application of \eqref{Trace_estimate_local} with $C_{\tr} \leq d_s^{-\frac{1}{2}}$ and the Cauchy-Schwarz inequality yield
\begin{equation}
\label{eq:efficiency_P}
\E_{P} (\bar V, \bar Z; \C_{z'}) \leq d_s^{-1} \|  \nabla e_{V} \|_{L^2(y^{\alpha},\C_{z'})} + d_s^{-1} \| \nabla e_{P} \|_{L^2( y^{\alpha}, \C_{z'})}.
\end{equation}

Step 3. The goal of this step is to analyze the efficiency properties of the indicator $\E_{Z}$ defined by \eqref{eq:defofEZglobal}--\eqref{eq:defofEZ}. A trivial application of the triangle inequality yields
\[
\E_{Z}(\bar{Z},\bar{P}; S_{z'}) \leq \| \bar Z - \Pi ( -\tfrac{1}{\mu} \tr \bar p)   \|_{L^2(S_{z'})} + \| \Pi ( -\tfrac{1}{\mu} \tr \bar p) -  \Pi ( - \tfrac{1}{\mu} \tr \bar P)   \|_{L^2(S_{z'})},
\]
where $\Pi$ denotes the nonlinear projector defined by \eqref{def:Pi}. Now, in view of the local Lipschitz continuity of $\Pi$, the fact that $\orsf = \Pi ( -\tfrac{1}{\mu} \tr \bar p)$ and the trace estimate \eqref{Trace_estimate_local} imply that
\begin{equation}
\label{eq:efficiency_Z}
 \E_{Z}(\bar{Z},\bar{P}; S_{z'}) \leq \| \orsf - \bar{Z} \|_{L^2(S_{z'})} + \frac{d_s^{-\frac{1}{2}}}{\mu}\| \nabla e_P \|_{L^2(y^{\alpha},\C_{z'})}.
\end{equation}

Step 4. The desired estimate \eqref{eq:efficiency} follows from a collection of the estimates \eqref{eq:efficiency_V}, \eqref{eq:efficiency_P}, and \eqref{eq:efficiency_Z}. This concludes the proof. 
\end{proof}

\begin{remark}[Local efficiency]\rm
  Examining the proof of Theorem \ref{th:global_efficiency}, we realize that the error indicators $\E_{V}$, $\E_{P}$ and $\E_{Z}$ are locally efficient; see inequalities \eqref{eq:efficiency_V}, \eqref{eq:efficiency_P} and \eqref{eq:efficiency_Z}, respectively. In addition, in all these inequalities the involved constants are known and depend only on the parameter $s$, through the constant $d_s$, and the parameter $\mu$. The key ingredients to derive the local efficiency property of the error estimator $\E_{Z}$ are the local Lipschitz continutiy of $\Pi$ and the trace estimate \eqref{Trace_estimate_local}. We comment that obtaining local a posteriori error bounds for the discretization of an optimal control problem is not always possible. We refer the reader to \cite[Remark~3.3]{KRS} for a thorough discussion on this matter. 
 \end{remark}

\subsection{A computable a posteriori error estimator}
\label{subsec:computable_a_posteriori}

The a posteriori error estimator proposed and analyzed in subsection \ref{subsec:ideal_a_posteriori} has an obvious drawback: given a node $z'$, its construction requires the knowledge of the functions $\zeta_{z'}$ and $\chi_{z'}$ that solve exactly the infinite--dimensional problems \eqref{eq:ideal_local_problemV} and \eqref{eq:ideal_local_problemP}, respectively. However, it provides intuition and sets the mathematical framework under which we will define a computable and anisotropic a posteriori error estimator. To describe it, we define the following discrete local spaces.

\begin{definition}[discrete local spaces]
\label{def:discrete_spaces}
For $z' \in \N(\T_\Omega)$, we define
\begin{align*}
  \mathcal{W}(\C_{z'}) &= 
    \left\{
      W \in C^0( \overline{\C_{z'}} ): W|_T \in \mathcal{P}_2(K) \otimes \mathbb{P}_2(I) \ \forall T = K \times I \in \C_{z'}, \right. \\
    &\left.  W|_{\partial \C_{z'} \setminus \Omega \times \{ 0\} } = 0
    \right\},
\end{align*}
where, if $K$ is a quadrilateral, $\mathcal{P}_2(K)$ stands for $\mathbb{Q}_2(K)$ --- the space of polynomials of degree not larger than $2$ in each variable.
If $K$ is a simplex, $\mathcal{P}_2(K)$ corresponds to $\mathbb{P}_2(K) \oplus \mathbb{B}(K)$ where where $\mathbb{P}_2(K)$ stands for the space of polynomials of total degree at most $2$, and $\mathbb{B}(K)$ is the space spanned by a local cubic bubble function.
\end{definition}

With these discrete spaces at hand, we proceed to define the computable counterpart of the error indicator $\E_{\mathrm{ocp}}$ given by \eqref{eq:defofEocp}. This indicator is defined as follows:
\begin{equation}
 \label{eq:defofEocp_computable}
 E_{\textrm{ocp}}(\bar{V},\bar{P},\bar{Z}; \T_{\Y}) = E_{V}(\bar{V},\bar{Z}; \N(\T_{\Omega})) + E_{P}(\bar{P},\bar{V}; \N(\T_{\Omega})) + E_{Z}(\bar{Z},\bar{P}; \T_{\Omega}),
\end{equation}
where $\T_{\Y} \in \Tr$ is the anisotropic mesh defined in subsection \ref{subsec:fully} and $\bar{V}$, $\bar{P}$ and $\bar{Z}$ denote the optimal variables solving the fully discrete optimal control problem. To describe the first contribution in \eqref{eq:defofEocp_computable}, we define $\eta_{z'} \in \Wcal(\C_{z'})$ as the solution to 
\begin{equation}
\label{eq:ideal_local_problemV_computable}
a_{z'}( \eta_{z'}, W)  = \langle \bar{Z}, \tr W  \rangle  -  a_{z'}(\bar{V}, W) \quad \forall W \in \Wcal(\C_{z'}).
\end{equation}
We then define the local and computable error estimator, for the state equation, as
\begin{equation}
\label{eq:defofEV_computable} 
E_{V}(\bar{V},\bar{Z}; \C_{z'}) := \| \nabla \eta_{z'} \|_{L^2(y^{\alpha},\C_{z'})},
\end{equation}
and the global error estimator 
$
E_{V}(\bar{V},\bar{Z}; \N(\T_{\Omega})) := \left( \sum_{z' \in \N(\T_{\Omega})} E_{V}^2(\bar{V},\bar{Z}; \C_{z'}) \right)^{\frac{1}{2}}.
$

The second contribution in \eqref{eq:defofEocp_computable} is defined on the basis of the discrete object $\theta_{z'} \in \Wcal(\C_{z'})$ that solves the following local problem:
\begin{equation}
\label{eq:ideal_local_problemP_computable}
a_{z'}( \theta_{z'}, W )= \langle \tr \bar{V} - \usf_d, \tr W \rangle  - a_{z'}( \bar{P}, W ) \quad \forall W \in \Wcal(\C_{z'}).
\end{equation}
We thus define the local and computable error indicator 
\begin{equation}
\label{eq:defofEP_computable} 
E_{P}(\bar{P},\bar{V}; \C_{z'}) := \| \nabla \theta_{z'} \|_{L^2(y^{\alpha},\C_{z'})}
\end{equation}
and the global error indicator $E_{P}(\bar{P},\bar{V}; \N( \T_{\Omega} ) ) := \left( \sum_{z' \in \N( \T_{\Omega} )} E_{P}^2(\bar{P},\bar{V}; \C_{z'}) \right)^{\frac{1}{2}}$.

The third contribution in \eqref{eq:defofEocp_computable}, \ie the error indicator associated to the optimal control $E_{Z}$, is defined by \eqref{eq:defofEZ}--\eqref{eq:defofEZglobal}.

We now explore the connection between the error estimator $E_{\textrm{ocp}}$ and the error. We first obtain a lower bound that does not involve any oscillation term. 

\begin{theorem}[local lower bound]
Let $(\bar{v}, \bar{p}, \orsf) \in \HL(y^{\alpha},\C_{\Y}) \times \HL(y^{\alpha},\C_{\Y}) \times \Zad$ be the solution to the optimality system associated with the truncated optimal control problem defined in subsection \ref{subsec:truncated} and $(\bar{V},\bar{P},\bar{Z}) \in \V(\T_{\Y}) \times \V(\T_{\Y}) \times \mathbb{Z}_{ad}(\T_{\Omega})$ its numerical approximation defined in subsection \ref{subsec:fully}. Then,
\begin{multline}
\label{eq:efficiency_computable}
 E_{V}(\bar{V},\bar{Z}; \C_{z'} ) + E_{P}(\bar{P},\bar{V}; \C_{z'} ) +  E_{Z}(\bar{Z},\bar{P}; S_{z'} ) \\ 
 \leq C(s,\mu) \left( \| \nabla( \bar{v} -\bar{V}) \|_{L^2(y^{\alpha},\C_{z'})}  +  \| \nabla( \bar{p} -\bar{P}) \|_{L^2(y^{\alpha},\C_{z'})} + \| \bar{\rsf} -\bar{Z} \|_{L^2(S_{z'})} \right),
\end{multline}
where $C(d_s,\mu)$ depends only on $d_s$ and the parameter $\mu$ and is defined in \eqref{eq:C}.
\label{th:global_efficiency_computable}
\end{theorem}
\begin{proof}
The proof of the estimate \eqref{eq:efficiency_computable} repeats the arguments developed in the proof of Theorem \ref{th:global_efficiency}. We analyze the local efficiency of the indicator $E_V$ defined in \eqref{eq:defofEV_computable}. To do this, we let $z' \in \N(\T_{\Omega})$. Employing the fact that $\eta_{z'}$ solves problem \eqref{eq:ideal_local_problemV_computable} and recalling that $\orsf$ denotes the optimal control, we arrive at 
\[
E_{V}^2(\bar V, \bar Z; \C_{z'}) = a_{z'}( \eta_{z'}, \eta_{z'}) = \langle \orsf, \eta_{z'}\rangle + \langle \bar Z - \orsf, \eta_{z'}\rangle - a_{z'}( \bar V, \eta_{z'}).
\]
Invoking the trace estimate \eqref{Trace_estimate_local} with constant $C_{\tr} \leq d_s^{-\frac{1}{2}}$, the fact that $\bar v$ solves problem \eqref{eq:alpha_harm_truncated}
and the Cauchy-Schwarz inequality, we obtain
\[
E_{V}^2(\bar V, \bar Z; \C_{z'}) \leq \left( d_s^{-1}\|  \nabla (\bar v - \bar V)\|_{L^2(y^{\alpha},\C_{z'})} + d_s^{-\frac{1}{2}}\| \orsf -\bar{Z} \|_{L^2(S_{z'})} \right) \|  \nabla \eta_{z'}\|_{L^2(y^{\alpha},\C_{z'})},
\]
which, in light of \eqref{eq:defofEV_computable}, immediately yields the desired result
\begin{equation*}
\label{eq:EV_locally_efficient}
E_{V}(\bar V, \bar Z; \C_{z'}) \leq d_s^{-1}\|  \nabla (\bar v - \bar V)\|_{L^2(y^{\alpha},\C_{z'})} + d_s^{-\frac{1}{2}}\| \orsf -\bar{Z} \|_{L^2(S_{z'})}.
\end{equation*}

The efficiency analysis for the terms $E_{P}$ and $E_{Z}$ follow similar arguments. For brevity we skip the proof.
\end{proof}

\begin{remark}[strong efficiency]
\rm
We remark that that the lower bound \eqref{eq:efficiency_computable} implies a strong concept of efficiency: it is free of any oscillation term and the involved constant $C(d_s,\mu)$ is known and given by \eqref{eq:C}. The relative size of the local error indicator dictates mesh refinement regardless of fine structure of the data. The analysis is valid for the family of anisotropic meshes $\T_\Y$ and allows the nonuniformly coefficients involved in problems \eqref{eq:alpha_harm_truncated} and \eqref{eq:p_truncated}.
\end{remark}

We now proceed to analyze the reliability properties of the anisotropic and computable error indicator $E_{\textrm{ocp}}$ defined in \eqref{eq:defofEocp_computable}. To achieve this, we introduce the so-called \emph{data oscillation}. Given a function $f \in L^2(\Omega)$ and $z' \in \N(\T_\Omega)$, we define the local oscillation of the function $f$ as
\begin{equation}
\label{eq:defoflocosc}
  \osc(f; S_{z'}) := h_{z'}^{s} \| f - f_{z'} \|_{L^2(S_{z'})},
\end{equation}
where, $h_{z'} = \min\{h_{K}: K \ni z' \}$ and $f_{z'}|_K \in \R$ is the average of $f$ over $K$, i.e., 
\begin{equation}
 f_{z'}|_K := \fint_{K} f.
\end{equation}
The global data oscillation is then defined as
\begin{equation}
\label{eq:defofglobosc}
  \osc(f;\T_\Omega) := \left( \sum_{z' \in \N(\T_\Omega)} \osc( f; S_{z'} )^2 \right)^{\frac{1}{2}}.
\end{equation}

To present our results in a concise manner, we define $D = (\usf_d, \tr \bar V)$ and
\begin{equation}
\label{eq:defofgloboscD}
  \osc(D ;S_{z'}) := \osc(\usf_d ;S_{z'})
  + \osc(\tr \bar V ;S_{z'}).
\end{equation}
We define $\osc(D ; \T_{\Omega})$ accordingly. We also define the \emph{total error indicator}
\begin{equation}
\label{total_error}
\mathcal{E}(\bar V, \bar P, \bar Z ; \C_{z'}) := \left( E_{\textrm{ocp}}(\bar V, \bar P, \bar Z; \C_{z'})^2 + \osc(D ;S_{z'})^2 \right)^{\frac{1}{2}} \quad 
\forall z' \in \N(\T_{\Omega}).
\end{equation}
This indicator will be used to mark elements for refinement in the AFEM proposed in section \ref{sec:numerics}. The following remark is then necessary.

\begin{remark}[marking] \rm
We comment that, in contrast to \cite{MR2421046}, the proposed AFEM will utilize the total error indicator, namely the sum of energy error and oscillation, for marking. This could be avoided if $E_{\textrm{ocp}}(\bar V, \bar P, \bar Z; \C_{z'}) \geq C \osc(\usf_d; S_{z'})$ for $C > 0$. While this property is trivial for the residual estimator with $C = 1$, it is in general false for other families of estimators such as the one we are proposing in this work. We refer to \cite{MR2421046} for a thorough discussion on this matter.
\end{remark}

Let $\mathscr{K}_{\T_{\Omega}} = \{ S_{z'} : z' \in \N(\T_\Omega)\}$ and, for any $\mathscr{M} \subset \mathscr{K}_{\T_{\Omega}}$, we set $\mathscr{M}_{\Y} = \mathscr{M} \times (0,\Y)$ and 
\begin{equation}\label{total_est}
\mathcal{E} ( \bar V, \bar P, \bar Z; \mathscr{M}_{\Y} ) := 
 \left( \sum_{S_{z'} \in \mathscr{M} } \mathcal{E}(\bar V, \bar P, \bar Z ; \C_{z'} )^2 \right)^{1/2},
\end{equation}
where, we recall that $\C_{z'} = S_{z'} \times (0,\Y)$. With these ingredients at hand, we present the following result.

\begin{theorem}[global upper bound]
Let $(\bar{v}, \bar{p}, \orsf) \in \HL(y^{\alpha},\C_{\Y}) \times \HL(y^{\alpha},\C_{\Y}) \times \Zad$ be the solution to the optimality system associated with the truncated optimal control problem defined in subsection \ref{subsec:truncated} and $(\bar{V},\bar{P},\bar{Z}) \in \V(\T_{\Y}) \times \V(\T_{\Y}) \times \mathbb{Z}_{ad}(\T_{\Omega})$ its numerical approximation defined in subsection \ref{subsec:fully}. If \eqref{condition} and \cite[Conjecture 5.28]{CNOS} hold, then
\begin{equation}
\| \nabla( \bar{v} -\bar{V}) \|_{L^2(y^{\alpha},\C_{\Y})}   +  \| \nabla( \bar{p} -\bar{P}) \|_{L^2(y^{\alpha},\C_{\Y})}
 + \| \bar{\rsf} -\bar{Z} \|_{L^2(\Omega)} \lesssim \mathcal{E}(\bar{V},\bar{P},\bar{Z}; \T_{\Y}),
 \label{eq:reliability_computable}
\end{equation}
where the hidden constant is independent of the continuous and discrete optimal variables and the size of the elements in the meshes $\T_{\Omega}$ and $\T_{\Y}$.
\label{th:reliability_computable}
\end{theorem}
\begin{proof}
The proof of the estimate \eqref{eq:reliability_computable} follows closely the arguments developed in the proof of Theorem \ref{th:reliability}; the difference being the use of the computable error indicator $E_{\textrm{ocp}}$ instead of the ideal estimator $\E_{\textrm{ocp}}$. We start by bounding the error in the control approximation. Defining $\tilde \rsf = \Pi (-\frac{1}{\mu} \tr \bar P)$, estimate \eqref{eq:r-Z-rtilde} implies that
\begin{equation}
\label{eq:r-Z-rtilde_computable}
 \| \orsf - \bar{Z} \|_{L^2(\Omega)} \leq  \| \orsf - \tilde{\rsf} \|_{L^2(\Omega)}  +  E_{Z}(\bar Z,\bar P; \T_{\Omega}).
\end{equation}
To control the remainder term, we invoke \eqref{eq:muleq} with $q$ defined by \eqref{eq:q_truncated} and write
\begin{equation}
\label{eq:muleq_computable}
\mu \| \orsf - \tilde{\rsf} \|^2_{L^2(\Omega)} \leq  (\tr(\bar{p} - q), \tilde{\rsf} - \orsf )_{L^2(\Omega)} + (\tr(q- \bar{P}), \tilde{\rsf} - \orsf )_{L^2(\Omega)} = \textrm{I} + \textrm{II}.
\end{equation}
To control the term $\textrm{II}$, we invoke the fact that $\bar P$, solution of problem \eqref{eq:P_discrete}, corresponds to the Galerkin approximation of $q$, solution of problem \eqref{eq:q_truncated}. This, in view of \cite[Theorem 5.37]{CNOS2}, yields
\begin{equation}
\label{eq:aux_computable}
 | \textrm{II} | \leq \frac{\mu}{4} \| \tilde \rsf - \orsf \|^2_{L^2(\Omega)} + C \left( E^2_{P}(\bar P, \bar V; \N(\T_{\Omega})) + \osc^2(D;\T_{\Omega})
\right), 
\end{equation}
where $C$ denotes a positive constant and $\osc$ is defined by \eqref{eq:defoflocosc} and \eqref{eq:defofglobosc}. 

To control the term $\mathrm{I}$, we write $\mathrm{I} = \mathrm{I}_1 + \mathrm{I}_2 := ( \tr(\bar{p} - w), \tilde \rsf - \orsf )_{L^2(\Omega)} + (\tr(w - q) , \tilde \rsf - \orsf )_{L^2(\Omega)}$, where $w$ is defined as the solution to \eqref{eq:w_truncated}. Step 3 in the proof of Theorem \ref{th:reliability} implies that $\mathrm{I}_1 \leq 0$. To control the term $\mathrm{I}_2$, we invoke \eqref{eq:aux1} and write 
\begin{equation}
\label{eq:aux1_computable}
 |\mathrm{I}_2| \lesssim \| \nabla ( w-q )  \|_{L^2(y^{\alpha},\C_{\Y})} \|\tilde{\rsf} - \orsf  \|_{L^2(\Omega)} \lesssim \|  \tr (\tilde{v}-\bar{V})  \|_{L^2(\Omega)} \|\tilde{\rsf} - \orsf  \|_{L^2(\Omega)}.
\end{equation}
We now write $\tilde{v}-\bar{V} = (\tilde{v}-v^*) - (v^* - \bar{V})$, where $v^*$ solves \eqref{eq:v_star}, and estimate each contribution separately. First, stability of \eqref{eq:v_star} yields
\begin{equation}
\label{eq:aux2_computable}
\|  \tr (\tilde{v}-v^*)  \|_{L^2(\Omega)} \lesssim E_{Z}(\bar{Z},\bar{P};\T_{\Omega}).
\end{equation}
Second, since $\bar V$ corresponds to the Galerkin approximation of $v^*$, \cite[Theorem 5.37]{CNOS2} implies the estimate
\begin{equation}
\label{eq:aux3_computable}
\|  \tr (v^*-\bar{V})  \|_{L^2(\Omega)} \lesssim \| \nabla( v^*-\bar{V} ) \|_{L^2(y^{\alpha},\C_{\Y})} \lesssim E_{V}(\bar{V},\bar{Z};\N(\T_{\Omega})). 
\end{equation}
This, in view of \eqref{eq:aux1_computable} and  \eqref{eq:aux2_computable}, implies that
\begin{align*}
 |\mathrm{I}_2| \leq & \frac{\mu}{4}\|\tilde{\rsf} - \orsf  \|^2_{L^2(\Omega)} + C \left( E^2_{V}(\bar{V},\bar{Z};\N(\T_{\Omega})) + E^2_{Z}(\bar{Z},\bar{P};\T_{\Omega}) \right),
\end{align*}
where $C$ denotes a positive constant. Since $\mathrm{I}_1 \leq 0$, a similar estimate holds for $\mathrm{I} = \mathrm{I}_1 + \mathrm{I}_2$. This estimate, in conjunction with the previous bound, and the estimates \eqref{eq:r-Z-rtilde_computable}, \eqref{eq:aux_computable} and \eqref{eq:muleq_computable} imply that
\begin{align*}
 \|\tilde{\rsf} - \orsf  \|^2_{L^2(\Omega)} \lesssim  E^2_{\textrm{ocp}}(\bar{V},\bar{P},\bar{Z};\T_{\Y}) + \osc^2(D;\T_{\Omega}). 
\end{align*}

The estimates for the terms $\|\nabla( \bar v - \bar V)  \|_{L^2(y^{\alpha},\C_{\Y})}$ and $\|\nabla( \bar p - \bar P)  \|_{L^2(y^{\alpha},\C_{\Y})}$ follow similar arguments to the ones elaborated in the steps 4 and 5 of the proof of Theorem \ref{th:reliability}. For brevity we skip the details.
\end{proof}

\begin{remark}[Conjecture 5.28 in \cite{CNOS}]\rm
Examining the proof of Theorem \ref{th:reliability_computable}, we realize that the key steps where \cite[Theorem 5.37]{CNOS2} is invoked are \eqref{eq:muleq_computable} and \eqref{eq:aux3_computable}. The results of \cite[Theorem 5.37]{CNOS2} are valid under the assumption of the existence of an operator $\mathcal{M}_{z'}$ that verify the conditions stipulated in \cite[Conjecture 5.28]{CNOS2}. The construction of the operator $\mathcal{M}_{z'}$ is an open problem. The numerical experiments of \cite[section 6]{CNOS} provide consistent computational evidence of the existence of $\mathcal{M}_{z'}$ with the requisite properties.
\end{remark}

\section{Numerical Experiments}
\label{sec:numerics}
In this section we conduct a numerical example that illustrates the performance of the proposed error estimator. To accomplish this task, we formulate an adaptive finite element method (AFEM) based on the following iterative loop:
\begin{equation}\label{eq:Aloop}
 \textsf{\textup{SOLVE}} \rightarrow  \textsf{\textup{ESTIMATE}} \rightarrow \textsf{\textup{MARK}} \rightarrow \textsf{\textup{REFINE}}
\end{equation}

\subsection{Design of AFEM}
\label{subsec:designofAfem}

We proceed to describe the four modules in \eqref{eq:Aloop}

\begin{enumerate}
 \item[$\bullet$] \textsf{\textup{SOLVE}:} Given $\T_{\Y}$, we compute $(\bar{Y},\bar{P},\bar{Z}) \in \V(\T_{\Y})\times \V(\T_{\Y})\times \mathbb{Z}_{ad}(\T_{\Omega})$, the solution to the fully discrete optimal control problem defined in subsection \ref{subsec:fully}:
 \[
  (\bar{Y},\bar{P},\bar{U}) = {\rm SOLVE}(\T_{\Y}).
 \]
To solve the minimization problem, we have used the projected BFGS method with Armijo line search; see \cite{MR1678201}. The optimization algorithm is terminated when the $\ell^2$-norm of the projected gradient is less or equal to $10^{-5}$.

 \item[$\bullet$] \textsf{\textup{ESTIMATE}:} Once a discrete solution is obtained, we compute, for each $z' \in \N(\T_{\Omega})$, the local error indicator \eqref{eq:defofEocp_computable}, which reads
\begin{equation*}
 E_{\textrm{ocp}}(\bar{V},\bar{P},\bar{Z}; \C_{z'}) = E_{V}(\bar{V},\bar{Z}; \C_{z'}) + E_{P}(\bar{P},\bar{V}; \C_{z'}) + E_{Z}(\bar{Z},\bar{P}; S_{z'}),
\end{equation*}
where the indicators $E_{V}$, $E_{P}$, $E_{Z}$ are defined by \eqref{eq:defofEV_computable}, \eqref{eq:defofEP_computable} and \eqref{eq:defofEZ}, respectively. We then compute the oscillation term \eqref{eq:defofgloboscD} and construct the total error indicator \eqref{total_error}:
\[
 \left\{ \mathcal{E} (\bar V, \bar P, \bar Z; S_{z'}) \right\}_{ S_{z'} \in \mathscr{K}_{\T_\Omega}} = \textsf{\textup{ESTIMATE}}(\bar V, \bar P, \bar Z; \T_{\Y}),
\]
where $\mathscr{K}_{\T_{\Omega}} = \{ S_{z'} : z' \in \N(\T_\Omega)\}$. For notational convenience, and in view of the fact that $\C_{z'} = S_{z'} \times (0,\Y)$ we replaced $\C_{z'}$ by $S_{z'}$ in the previous formula.
\item[$\bullet$] \textsf{\textup{MARK}:} Using the so--called  D\"{o}rfler marking strategy \cite{MR1393904} (bulk chasing strategy) with parameter $\theta$ with $\theta \in (0,1]$, we select a set 
\[
  \mathscr{M} = \textsf{\textup{MARK}} \left( \left\{ \mathcal{E} (\bar V, \bar P, \bar Z; S_{z'}) \right\}_{ S_{z'} \in \mathscr{K}_{\T_\Omega}}, (\bar V, \bar P, \bar Z) \right) \subset \mathscr{K}_{\T_{\Omega}}
\]
of minimal cardinality that satisfies
\[
  \mathcal{E} ( (\bar V, \bar P, \bar Z),\mathscr{M})  \geq \theta \mathcal{E} ((\bar V, \bar P, \bar Z),  \mathscr{K}_{\T_{\Omega}}). 
\]
\item[$\bullet$] \textsf{\textup{REFINEMENT}:} We generate a new mesh $\T_\Omega'$ by bisecting all the elements $K \in \T_{\Omega}$ contained in $\mathscr{M}$ based on the newest-vertex bisection method \cite{NSV:09,NV}. We choose the truncation parameter as $\Y = 1 + \tfrac{1}{3}\log(\# \T_{\Omega}')$ to balance the approximation and truncation errors \cite[Remark 5.5]{NOS}. The mesh $\mathcal{I}_\Y'$ is constructed by the rule \eqref{eq:graded_mesh}, with a number of degrees of freedom $M$ sufficiently large so that \eqref{condition} holds. This is attained by first creating a partition $\mathcal{I}_\Y'$ with $M \approx (\# \T_\Omega')^{1/n}$
and checking \eqref{condition}. If this condition is violated, we increase 
the number of points until we get the desired result. The new mesh
\[
  \T_\Y' = \textsf{\textup{REFINE}}(\mathscr{M}),
\]
is obtained as the tensor product of $\T_\Omega'$ and $\mathcal{I}_\Y'$.
\end{enumerate}

\subsection{Implementation}
\label{subsec:implementation}

The AFEM \eqref{eq:Aloop} is implemented within the MATLAB software library {\it{i}}FEM~\cite{Chen.L2008c}. All matrices have been assembled exactly. The right hand sides are  computed by a quadrature formula which is exact for polynomials of degree $4$.  All linear systems were solved using he multigrid method with line smoother introduced and analyzed in~\cite{CNOS}.

To compute the solution $\eta_{z'}$ to the discrete local problem \eqref{eq:ideal_local_problemV_computable} we proceed as follow: we loop around each node $z' \in \N(\T_{\Omega})$, 
collect data about the cylindrical star $\C_{z'}$ and assemble the small linear
system \eqref{eq:ideal_local_problemV_computable}.  This linear system is solved by the built-in \emph{direct solver} of MATLAB. To compute the solution $\theta_{z'}$ to the discrete local problem \eqref{eq:ideal_local_problemP_computable}, we proceed similarly. All integrals involving only the weight and discrete functions
are computed exactly, whereas those also involving data functions are computed
element-wise by a quadrature formula which is exact for polynomials of degree 7. 

For convenience, in the \textsf{\textup{MARK}} step we change the estimator from star--wise to element--wise. To acomplish this task, we first scale the nodal-wise estimator as $E_{\textrm{ocp}}^2(\bar V, \bar P, \bar Z; \C_{z'}) / (\# S_{z'} )$ and then, for each element $K \in \T_{\Omega}$, we compute 
\[
E_{\textrm{ocp}}^2(\bar V, \bar P, \bar Z; K_{\Y}) := \sum_{z'\in K} E_{\textrm{ocp}}^2(\bar V, \bar P, \bar Z; \C_{z'}), 
\]
where $K_{\Y} = K \times (0,\Y)$. The scaling is introduced so that  
\[
\sum _{K \in \T_{\Omega} }E_{\textrm{ocp}}^2(\bar V, \bar P, \bar Z; K_{\Y}) = \sum _{z' \in \N(\T_\Omega) } E_{\textrm{ocp}}^2(\bar V, \bar P, \bar Z; \C_{z'}).         \]
The cell-wise data oscillation is now defined as
\[
\osc (f;K)^2 := h_{K}^{2s}\|f - \bar f_K\|^2_{L^2(K)}, 
\]
where $\bar f_K$ denotes the average of $f$ over the element $K$. This quantity is computed using a quadrature formula which is exact for polynomials of degree 7.

\subsection{L-shaped domain with incompatible data}
\label{sub:Lshapedcompatible}

For our numerical example we consider the worst possible scenario:
\begin{enumerate}[(\textsf{D}1)]
\item \label{incomp_ab} $\asf = 0.1$, $\bsf = 0.3$. This implies that the optimal control $\bar{\zsf} \not \in \mathbb{H}^{1-s}(\Omega)$ when $s < \frac{1}{2}$. We will refer to the optimal control $\ozsf$ as \emph{incompatible datum} for \eqref{eq:alpha_harm_truncated}.
\item \label{incom_ud} $\usf_d = 1$. This element does not belong to $\mathbb{H}^{1-s}(\Omega)$ when $s < \frac{1}{2}$. Therefore, for $s < \frac{1}{2}$, $\usf_d$ is an \emph{incompatible datum} for problem \eqref{eq:p_truncated}.
\item \label{nonconvex} $\Omega = (-1,1)^2 \setminus (0,1) \times (-1,0)$, i.e., the well known L-shaped domain; see Figure~\ref{f:lshape}. 
\end{enumerate}

In view of (\textsf{D}\ref{incomp_ab}) and (\textsf{D}\ref{incom_ud}), we conclude that the right hand sides of the state and adjoint equations, problems \eqref{eq:alpha_harm_truncated} and \eqref{eq:p_truncated}, respectively, are incompatible for $s < \frac{1}{2}$. As discussed in \cite[section~6.3]{NOS}, at the level of the state equation, this results in lower rates of convergence when quasi-uniform refinement of $\Omega$ is employed. In addition, we consider a situation where the domain $\Omega$ is noncovex.
As a result, the hypothesis of Theorem~\ref{fd2} does not hold and then it cannot be applied.

We set $\mu = 1$, and we comment that we do not explicitly enforce the mesh restriction \eqref{condition}, which shows that this is nothing but an artifact in our theory. 

As Figure~\ref{f:LshapeEst} illustrates, using our proposed AFEM driven by the error indicator \eqref{total_error}, we can recover the optimal rates of convergence \eqref{fd2}--\eqref{fd1} for all values of $s$ considered: $s = 0.2, 0.4, 0.6$, and $s = 0.8$. We remark, again, that we are operating under the conditions (\textsf{D}\ref{incomp_ab})--(\textsf{D}\ref{nonconvex}) and then Theorem~\ref{fd2} cannot be applied. Since, for $s<\frac{1}{2}$, the data is incompatible (\textsf{D}\ref{incomp_ab})--(\textsf{D}\ref{incom_ud}), the optimal and adjoint states exhibit boundary layers. To capture them, our AFEM refines near the boundary; see Figure~\ref{f:lshape} (middle). In contrast, when $s > \frac{1}{2}$ such incompatibilities does not occur and then our AFEM focuses to resolve the reentrant corner; see Figure~\ref{f:lshape} (right). The left panel in Figure~\ref{f:lshape} shows the initial mesh. We comment that the middle and the right panels are obtained with 17 AFEM cycles.

\begin{figure}[h!]
\centering
\includegraphics[width=0.6\textwidth]{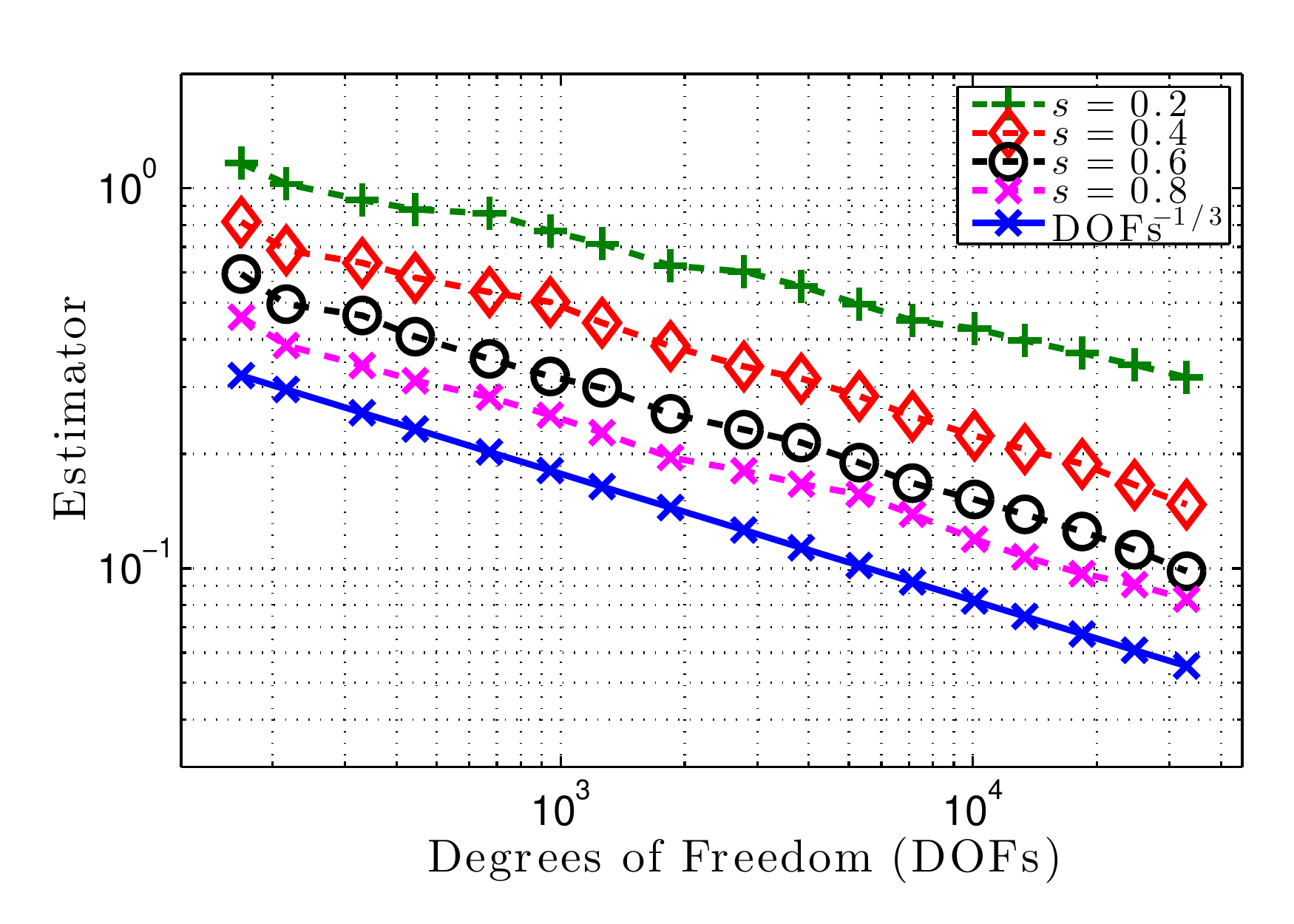}
\caption{\label{f:LshapeEst}
Computational rate of convergence for our anisotropic AFEM with incompatible right hand sides for both the state equation and the adjoint equation over an L--shaped domain (non-convex domain). We consider $n = 2$. Since the exact solution is not known for this problem, we present the total error estimator with respect to the number of degrees of freedom. In all cases we recover the optimal rate of convergence $(\# \T_{\Y})^{-1/3}$.}
\end{figure}

\begin{figure}[h!]
\centering
\includegraphics[width=0.31\textwidth]{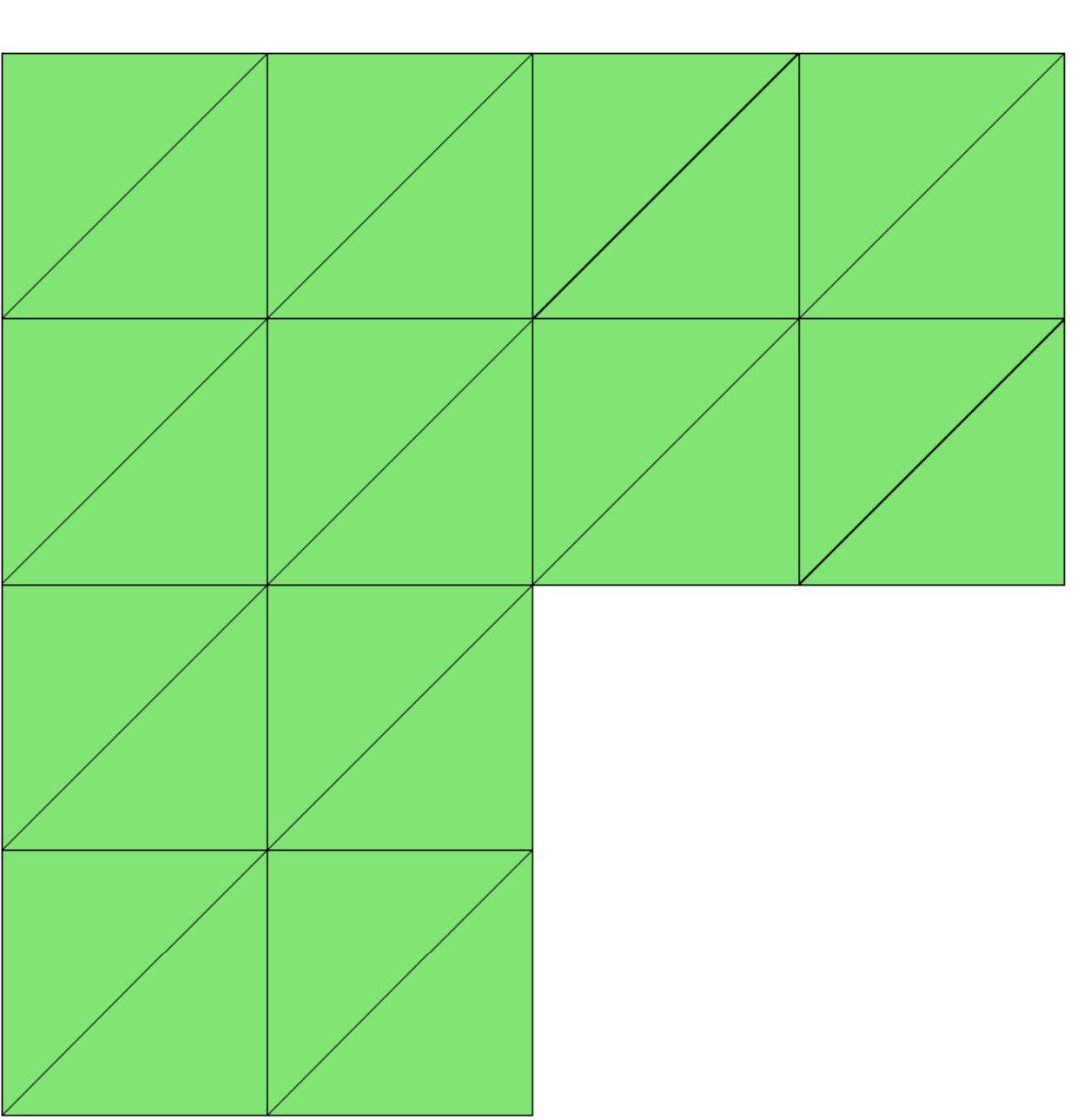}\quad
\includegraphics[width=0.31\textwidth]{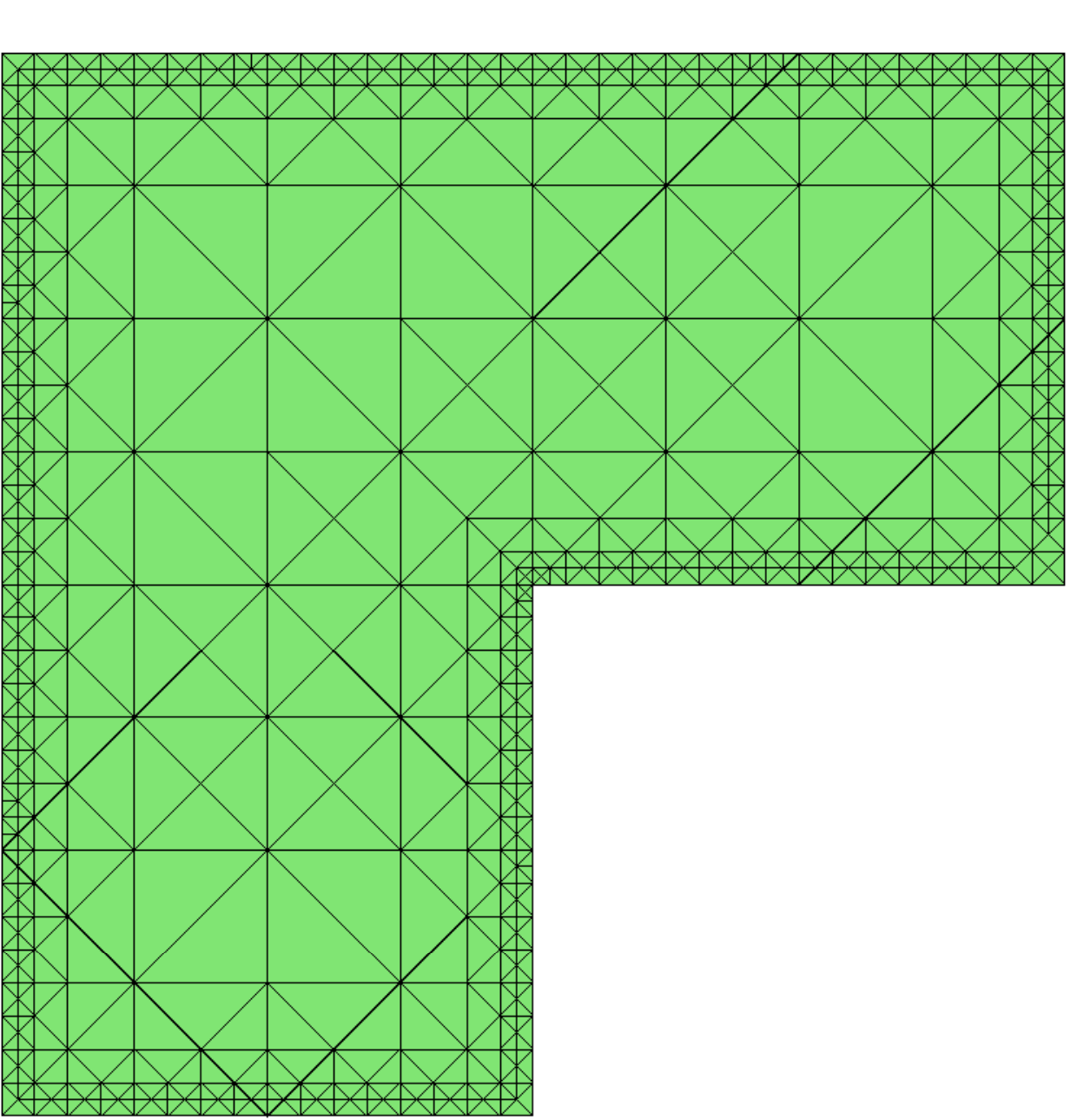}\quad
\includegraphics[width=0.31\textwidth]{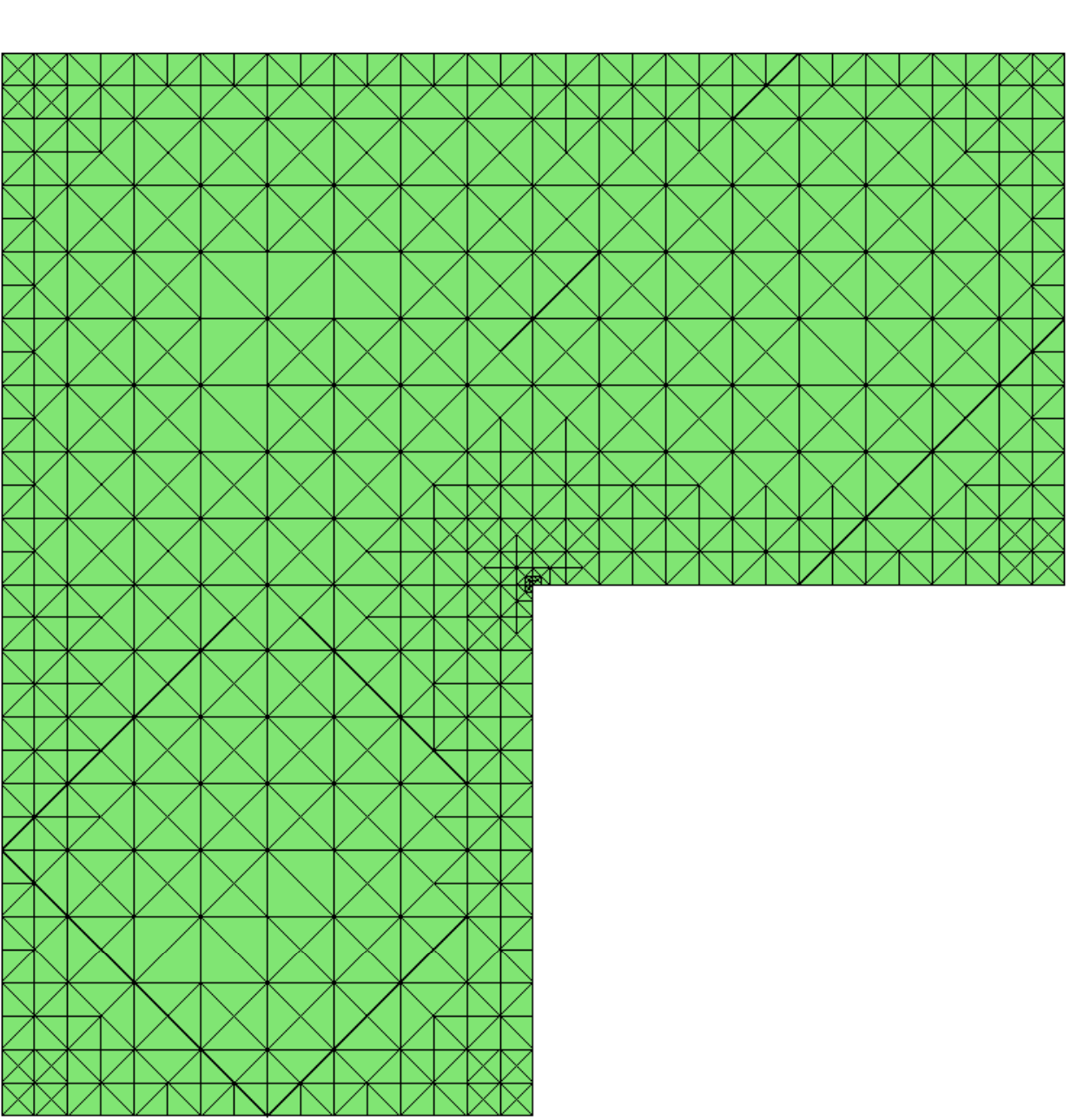}\quad
\caption{\label{f:lshape} The left panel shows the initial grid. The middle and right panels shows adaptive grids, obtained after 17 refinements, for $s = 0.2$ and $s = 0.8$, respectively. We consider an L-shaped domain with incompatible right hand side for the state and adjoint equations. As expected when $s < \frac{1}{2}$ the incompatible data ( $\ozsf, \usf_d \notin \Ws$) results in boundary layers for both the state and the adjoint state. In order to capture them, our AFEM refines near the boundary. In contrast, when $s > \frac{1}{2}$ the refinement is more pronounced near the reentrant corner; the data $\bar{\zsf}$ and $\usf_d$ are compatible in this case.}
\end{figure}

\bibliographystyle{plain}
\bibliography{biblio}

\def\cprime{$'$} \def\cprime{$'$} \def\cprime{$'$} \def\cprime{$'$}
  \def\cprime{$'$} \def\cprime{$'$}
\begin{thebibliography}{10}

\bibitem{Abe2005403}
S.~Abe and S.~Thurner.
\newblock Anomalous diffusion in view of {E}instein's 1905 theory of {B}rownian
  motion.
\newblock {\em Physica A: Statistical Mechanics and its Applications},
  356(2–4):403 -- 407, 2005.

\bibitem{Abra}
M.~Abramowitz and I.A. Stegun.
\newblock {\em Handbook of mathematical functions with formulas, graphs, and
  mathematical tables}, volume~55 of {\em National Bureau of Standards Applied
  Mathematics Series}.
\newblock 1964.

\bibitem{AObook}
M.~Ainsworth and J.T. Oden.
\newblock {\em A posteriori error estimation in finite element analysis}.
\newblock Pure and Applied Mathematics (New York). Wiley-Interscience, New
  York, 2000.

\bibitem{AO}
H.~Antil and E.~Ot\'arola.
\newblock A {FEM} for an optimal control problem of fractional powers of
  elliptic operators.
\newblock {\em SIAM J. Control Optim.}, 53(6):3432--3456, 2015.

\bibitem{ACT:02}
N.~Arada, E.~Casas, and F.~Tr{\"o}ltzsch.
\newblock Error estimates for the numerical approximation of a semilinear
  elliptic control problem.
\newblock {\em Comput. Optim. Appl.}, 23(2):201--229, 2002.

\bibitem{atanackovic2014fractional}
T.M. Atanackovic, S.~Pilipovic, B.~Stankovic, and D.~Zorica.
\newblock {\em Fractional Calculus with Applications in Mechanics: Vibrations
  and Diffusion Processes}.
\newblock John Wiley \& Sons, 2014.

\bibitem{MR880421}
I.~Babu{\v{s}}ka and A.~Miller.
\newblock A feedback finite element method with a posteriori error estimation.
  {I}. {T}he finite element method and some basic properties of the a
  posteriori error estimator.
\newblock {\em Comput. Methods Appl. Mech. Engrg.}, 61(1):1--40, 1987.

\bibitem{BR78}
I.~Babu{\v{s}}ka and W.C. Rheinboldt.
\newblock Error estimates for adaptive finite element computations.
\newblock {\em SIAM J. Numer. Anal.}, 15(4):736--754, 1978.

\bibitem{BSbook}
I.~Babu{\v{s}}ka and T.~Strouboulis.
\newblock {\em The Finite Element Method and Its Reliability}.
\newblock Numerical mathematics and scientific computation. Clarendon Press,
  2001.

\bibitem{MR1736459}
E.~Barkai, R.~Metzler, and J.~Klafter.
\newblock From continuous time random walks to the fractional {F}okker-{P}lanck
  equation.
\newblock {\em Phys. Rev. E (3)}, 61(1):132--138, 2000.

\bibitem{MR1780911}
R.~Becker, H.~Kapp, and R.~Rannacher.
\newblock Adaptive finite element methods for optimal control of partial
  differential equations: basic concept.
\newblock {\em SIAM J. Control Optim.}, 39(1):113--132 (electronic), 2000.

\bibitem{MR1081295}
J.-P. Bouchaud and A.~Georges.
\newblock Anomalous diffusion in disordered media: statistical mechanisms,
  models and physical applications.
\newblock {\em Phys. Rep.}, 195(4-5):127--293, 1990.

\bibitem{BV}
C.~Bucur and E.~Valdinoci.
\newblock Nonlocal diffusion and applications.
\newblock arXiv:1504.08292, 2015.

\bibitem{bio}
A.~Bueno-Orovio, D.~Kay, V.~Grau, B.~Rodriguez, and K.~Burrage.
\newblock Fractional diffusion models of cardiac electrical propagation: role
  of structural heterogeneity in dispersion of repolarization.
\newblock {\em J. R. Soc. Interface}, 11(97), 2014.

\bibitem{CS:07}
L.~Caffarelli and L.~Silvestre.
\newblock An extension problem related to the fractional {L}aplacian.
\newblock {\em Comm. Part. Diff. Eqs.}, 32(7-9):1245--1260, 2007.

\bibitem{CDDS:11}
A.~Capella, J.~D{\'a}vila, L.~Dupaigne, and Y.~Sire.
\newblock Regularity of radial extremal solutions for some non-local semilinear
  equations.
\newblock {\em Comm. Part. Diff. Eqs.}, 36(8):1353--1384, 2011.

\bibitem{CF:00}
C.~Carstensen and S.A. Funken.
\newblock Fully reliable localized error control in the {FEM}.
\newblock {\em SIAM J. Sci. Comput.}, 21(4):1465--1484 (electronic), 1999/00.

\bibitem{CT:05}
E.~Casas, M.~Mateos, and F.~Tr{\"o}ltzsch.
\newblock Error estimates for the numerical approximation of boundary
  semilinear elliptic control problems.
\newblock {\em Comput. Optim. Appl.}, 31(2):193--219, 2005.

\bibitem{MR2421046}
J.M. Casc\'on, C.~Kreuzer, R.H. Nochetto, and K.G. Siebert.
\newblock Quasi-optimal convergence rate for an adaptive finite element method.
\newblock {\em SIAM J. Numer. Anal.}, 46(5):2524--2550, 2008.

\bibitem{MR2875241}
J.M. Casc{\'o}n and R.H. Nochetto.
\newblock Quasioptimal cardinality of {AFEM} driven by nonresidual estimators.
\newblock {\em IMA J. Numer. Anal.}, 32(1):1--29, 2012.

\bibitem{Chen.L2008c}
L.~Chen.
\newblock {$i$FEM}: An integrated finite element methods package in matlab.
\newblock Technical report, University of California at Irvine, 2009.

\bibitem{CNOS2}
L.~Chen, R.~H. Nochetto, E.~Ot{\'a}rola, and A.~J. Salgado.
\newblock A {PDE} approach to fractional diffusion: a posteriori error
  analysis.
\newblock {\em J. Comput. Phys.}, 293:339--358, 2015.

\bibitem{CNOS}
L.~Chen, R.H. Nochetto, E.~Ot\'arola, and A.J. Salgado.
\newblock Multilevel methods for nonuniformly elliptic operators and fractional
  diffusion.
\newblock {\em Math. Comp.}, 2016.
\newblock (to appear).

\bibitem{wow}
W.~Chen.
\newblock A speculative study of $2/3$-order fractional laplacian modeling of
  turbulence: Some thoughts and conjectures.
\newblock {\em Chaos}, 16(2):1--11, 2006.

\bibitem{CiarletBook}
P.G. Ciarlet.
\newblock {\em The finite element method for elliptic problems}, volume~40 of
  {\em Classics in Applied Mathematics}.
\newblock SIAM, Philadelphia, PA, 2002.

\bibitem{MR2035411}
L.~Debnath.
\newblock Fractional integral and fractional differential equations in fluid
  mechanics.
\newblock {\em Fract. Calc. Appl. Anal.}, 6(2):119--155, 2003.

\bibitem{MR2025566}
L.~Debnath.
\newblock Recent applications of fractional calculus to science and
  engineering.
\newblock {\em Int. J. Math. Math. Sci.}, (54):3413--3442, 2003.

\bibitem{NEGRETE}
D.~del Castillo-Negrete, B.~A. Carreras, and V.~E. Lynch.
\newblock Fractional diffusion in plasma turbulence.
\newblock {\em Physics of Plasmas}, 11(8):3854--3864, 2004.

\bibitem{MR1393904}
W.~D{\"o}rfler.
\newblock A convergent adaptive algorithm for {P}oisson's equation.
\newblock {\em SIAM J. Numer. Anal.}, 33(3):1106--1124, 1996.

\bibitem{Javier}
J.~Duoandikoetxea.
\newblock {\em Fourier analysis}, volume~29 of {\em Graduate Studies in
  Mathematics}.
\newblock American Mathematical Society, Providence, RI, 2001.

\bibitem{DL:05}
R.G. Dur{\'a}n and A.L. Lombardi.
\newblock Error estimates on anisotropic {$Q_1$} elements for functions in
  weighted {S}obolev spaces.
\newblock {\em Math. Comp.}, 74(252):1679--1706 (electronic), 2005.

\bibitem{MR1865506}
L.~Formaggia and S.~Perotto.
\newblock New anisotropic a priori error estimates.
\newblock {\em Numer. Math.}, 89(4):641--667, 2001.

\bibitem{MR1971213}
L.~Formaggia and S.~Perotto.
\newblock Anisotropic error estimates for elliptic problems.
\newblock {\em Numer. Math.}, 94(1):67--92, 2003.

\bibitem{GH:14}
P.~Gatto and J.S. Hesthaven.
\newblock Numerical approximation of the fractional laplacian via hp-finite
  elements, with an application to image denoising.
\newblock {\em J. Sci. Comp.}, 65(1):249--270, 2015.

\bibitem{GU}
V.~Gol{\cprime}dshtein and A.~Ukhlov.
\newblock Weighted {S}obolev spaces and embedding theorems.
\newblock {\em Trans. Amer. Math. Soc.}, 361(7):3829--3850, 2009.

\bibitem{MR1926470}
R.~Gorenflo, F.~Mainardi, D.~Moretti, and P.~Paradisi.
\newblock Time fractional diffusion: a discrete random walk approach.
\newblock {\em Nonlinear Dynam.}, 29(1-4):129--143, 2002.
\newblock Fractional order calculus and its applications.

\bibitem{Grisvard}
P.~Grisvard.
\newblock {\em Elliptic problems in nonsmooth domains}, volume~24 of {\em
  Monographs and Studies in Mathematics}.
\newblock Pitman (Advanced Publishing Program), Boston, MA, 1985.

\bibitem{HHIK}
M.~Hinterm\"{u}ller, R.~H.~W. Hoppe, Y.~Iliash, and M.~Kieweg.
\newblock An a posteriori error analysis of adaptive finite element methods for
  distributed elliptic control problems with control constraints.
\newblock {\em ESAIM: Control Optim. Calc. of Var.}, 14:540--560, 7 2008.

\bibitem{Hinze:05}
M.~Hinze.
\newblock A variational discretization concept in control constrained
  optimization: the linear-quadratic case.
\newblock {\em Comput. Optim. Appl.}, 30(1):45--61, 2005.

\bibitem{ICH}
R.~Ishizuka, S.-H. Chong, and F.~Hirata.
\newblock An integral equation theory for inhomogeneous molecular fluids: The
  reference interaction site model approach.
\newblock {\em J. Chem. Phys}, 128(3), 2008.

\bibitem{MR1678201}
C.~T. Kelley.
\newblock {\em Iterative methods for optimization}, volume~18 of {\em Frontiers
  in Applied Mathematics}.
\newblock Society for Industrial and Applied Mathematics (SIAM), Philadelphia,
  PA, 1999.

\bibitem{KRS}
K.~Kohls, A.~R\"{o}sch, and K.G. Siebert.
\newblock A posteriori error analysis of optimal control problems with control
  constraints.
\newblock {\em SIAM J. Control Optim.}, 52(3):1832--1861, 2014.

\bibitem{MR1785418}
G.~Kunert.
\newblock An a posteriori residual error estimator for the finite element
  method on anisotropic tetrahedral meshes.
\newblock {\em Numer. Math.}, 86(3):471--490, 2000.

\bibitem{MR1777490}
G.~Kunert and R.~Verf{\"u}rth.
\newblock Edge residuals dominate a posteriori error estimates for linear
  finite element methods on anisotropic triangular and tetrahedral meshes.
\newblock {\em Numer. Math.}, 86(2):283--303, 2000.

\bibitem{Landkof}
N.S. Landkof.
\newblock {\em Foundations of modern potential theory}.
\newblock Springer-Verlag, New York, 1972.
\newblock Translated from the Russian by A. P. Doohovskoy, Die Grundlehren der
  mathematischen Wissenschaften, Band 180.

\bibitem{MR2064019}
S.~Z. Levendorski{\u\i}.
\newblock Pricing of the {A}merican put under {L}\'evy processes.
\newblock {\em Int. J. Theor. Appl. Finance}, 7(3):303--335, 2004.

\bibitem{MR2837575}
S.~Micheletti and S.~Perotto.
\newblock The effect of anisotropic mesh adaptation on {PDE}-constrained
  optimal control problems.
\newblock {\em SIAM J. Control Optim.}, 49(4):1793--1828, 2011.

\bibitem{MR1770058}
P.~Morin, R.~H. Nochetto, and K.~G. Siebert.
\newblock Data oscillation and convergence of adaptive {FEM}.
\newblock {\em SIAM J. Numer. Anal.}, 38(2):466--488 (electronic), 2000.

\bibitem{MNS02}
P.~Morin, R.H. Nochetto, and K.G. Siebert.
\newblock Local problems on stars: a posteriori error estimators, convergence,
  and performance.
\newblock {\em Math. Comp.}, 72(243):1067--1097 (electronic), 2003.

\bibitem{Muckenhoupt}
B.~Muckenhoupt.
\newblock Weighted norm inequalities for the {H}ardy maximal function.
\newblock {\em Trans. Amer. Math. Soc.}, 165:207--226, 1972.

\bibitem{PSSB:PSSB2221330150}
R.R. Nigmatullin.
\newblock The realization of the generalized transfer equation in a medium with
  fractal geometry.
\newblock {\em Physica Status Solidi (b)}, 133(1):425--430, 1986.

\bibitem{NOS2}
R.~H. Nochetto, E.~Ot{\'a}rola, and A.~J. Salgado.
\newblock Piecewise polynomial interpolation in {M}uckenhoupt weighted
  {S}obolev spaces and applications.
\newblock {\em Numer. Math.}, 132(1):85--130, 2016.

\bibitem{NOS}
R.H. Nochetto, E.~Ot\'arola, and A.~J. Salgado.
\newblock A {PDE} approach to fractional diffusion in general domains: A priori
  error analysis.
\newblock {\em Found. Comput. Math.}, 15(3):733--791, 2015.

\bibitem{NOS3}
R.H. Nochetto, E.~Ot\'arola, and A.J. Salgado.
\newblock A {PDE} approach to space-time fractional parabolic problems.
\newblock {\em SIAM J. Numer. Anal.}, 54(2):848--873, 2016.

\bibitem{NSV:09}
R.H. Nochetto, K.G. Siebert, and A.~Veeser.
\newblock Theory of adaptive finite element methods: an introduction.
\newblock In {\em Multiscale, nonlinear and adaptive approximation}. Springer,
  2009.

\bibitem{NV}
R.H. Nochetto and A.~Vesser.
\newblock Primer of adaptive finite element methods.
\newblock In {\em Multiscale and Adaptivity: Modeling, Numerics and
  Applications, CIME Lectures}. Springer, 2011.

\bibitem{MR2257635}
M.~Picasso.
\newblock Anisotropic a posteriori error estimate for an optimal control
  problem governed by the heat equation.
\newblock {\em Numer. Methods Partial Differential Equations},
  22(6):1314--1336, 2006.

\bibitem{10.2307/2007953}
A.~Weiser R.~E.~Bank.
\newblock Some a posteriori error estimators for elliptic partial differential
  equations.
\newblock {\em Math. Comp.}, 44(170):283--301, 1985.

\bibitem{MR1604710}
A.I. Saichev and G.M. Zaslavsky.
\newblock Fractional kinetic equations: solutions and applications.
\newblock {\em Chaos}, 7(4):753--764, 1997.

\bibitem{MR1389492}
K.~G. Siebert.
\newblock An a posteriori error estimator for anisotropic refinement.
\newblock {\em Numer. Math.}, 73(3):373--398, 1996.

\bibitem{Silvestre:2007}
L.~Silvestre.
\newblock Regularity of the obstacle problem for a fractional power of the
  {L}aplace operator.
\newblock {\em Comm. Pure Appl. Math.}, 60(1):67--112, 2007.

\bibitem{Stein}
E.M. Stein.
\newblock {\em Singular integrals and differentiability properties of
  functions}.
\newblock Princeton Mathematical Series, No. 30. Princeton University Press,
  Princeton, N.J., 1970.

\bibitem{ST:10}
P.R. Stinga and J.L. Torrea.
\newblock Extension problem and {H}arnack's inequality for some fractional
  operators.
\newblock {\em Comm. Part. Diff. Eqs.}, 35(11):2092--2122, 2010.

\bibitem{Turesson}
B.O. Turesson.
\newblock {\em Nonlinear potential theory and weighted {S}obolev spaces},
  volume 1736 of {\em Lecture Notes in Mathematics}.
\newblock Springer-Verlag, Berlin, 2000.

\bibitem{Verfurth}
R.~Verf\"{u}rth.
\newblock {\em A Review of A Posteriori Error Estimation and Adaptive
  Mesh-Refinement Techniques}.
\newblock John Wiley, 1996.

\end{thebibliography}

\end{document}